\documentclass{amsart}
\usepackage{amssymb,amsbsy,amsthm,amsmath,graphicx,epsfig,mathrsfs}

\usepackage{setspace}
\usepackage{version}

\numberwithin{equation}{section}

\theoremstyle{plain}
\newtheorem{theorem}[subsection]{Theorem}
\newtheorem{proposition}[subsection]{Proposition}
\newtheorem{lemma}[subsection]{Lemma}
\newtheorem{corollary}[subsection]{Corollary}
\newtheorem{conjecture}[subsection]{Conjecture}

\newtheorem*{conj3-repeat}{Conjecture \ref{conj3}}
\newtheorem*{conjecture-1.1-repeat}{Conjecture \ref{conjecture-1.1}}
\newtheorem*{interval-sieve-repeat}{Theorem \ref{interval-sieve}}
\newtheorem*{conj1-repeat}{Conjecture \ref{conj1}}
\newtheorem*{stability-thm-repeat}{Theorem \ref{stability-thm}}
\newtheorem*{stability-thm-symmetric-repeat}{Theorem \ref{stability-thm-symmetric}}
\newtheorem*{stability-conj-repeat}{Conjecture \ref{stability-conj}}
\newtheorem*{ils-ostmann-repeat}{Theorem \ref{ils-ostmann}}
\newtheorem*{prog-thm-rpt}{Theorem \ref{prog-thm}}

\theoremstyle{definition}

\theoremstyle{remark}

\renewcommand{\leq}{\leqslant}
\renewcommand{\geq}{\geqslant}
\newsavebox{\proofbox}

\savebox{\proofbox}{\begin{picture}(7,7)  \put(0,0){\framebox(7,7){}}\end{picture}}

\newcommand{\md}[1]{\ensuremath{(\operatorname{mod}\, #1)}}
\newcommand{\mdsub}[1]{\ensuremath{(\mbox{\scriptsize mod}\, #1)}}
\newcommand{\mdlem}[1]{\ensuremath{(\mbox{\textup{mod}}\, #1)}}
\newcommand{\mdsublem}[1]{\ensuremath{(\mbox{\scriptsize \textup{mod}}\, #1)}}

\newcommand\Z{\mathbb{Z}}
\newcommand\R{\mathbb{R}}

\newcommand\N{\mathbb{N}}
\newcommand\A{\mathscr{A}}
\newcommand\B{\mathscr{B}}

\newcommand\unif{\operatorname{unif}}

\newcommand\Q{\mathbb{Q}}

\newcommand\eps{\varepsilon}

\onehalfspace
\begin{document}

\title{Inverse questions for the large sieve}

\author[Green]{Ben Green}
\address{Mathematical Institute\\
Radcliffe Observatory Quarter\\
Woodstock Road\\
Oxford OX2 6GG\\
England }
\email{ben.green@maths.ox.ac.uk}

\author[Harper]{Adam J Harper}
\address{Jesus College \\
Cambridge CB5 8BL\\
England }
\email{A.J.Harper@dpmms.cam.ac.uk}

\begin{abstract} 
Suppose that an infinite set $\A$ occupies at most $\frac{1}{2}(p+1)$ residue classes modulo $p$, for every sufficiently large prime $p$. The squares, or more generally the integer values of any quadratic, are an example of such a set. By the large sieve inequality the number of elements of $\A$ that are at most $X$ is $O(X^{1/2})$, and the quadratic examples show that this is sharp. The simplest form of the \emph{inverse large sieve problem} asks whether they are the only examples. We prove a variety of results and formulate various conjectures in connection with this problem, including several improvements of the large sieve bound when the residue classes occupied by $\A$ have some additive structure. Unfortunately we cannot solve the problem itself.
\end{abstract}

\maketitle

\begin{center}\emph{To Roger Heath-Brown on his 60th birthday}\end{center}

\setcounter{tocdepth}{1}	

\tableofcontents

\section{Introduction}

\textsc{Notation.} Most of our notation is quite standard. When dealing with infinite sets $\A$, we write $\A[X]$ for the intersection of $\A$ with the initial segment $[X] := \{1,\dots, X\}$.

Our primary aim in this paper is to study sets $\A$ of integers with the property that the reduction $\A \md{p}$ occupies at most $\frac{1}{2}(p+1)$ residue classes modulo $p$ for all sufficiently large primes $p$. It follows from the large sieve that $|\A[X]| \ll X^{1/2}$ for all $X$ (we will recall the details of this argument below). This is clearly sharp up to the value of the implied constant, as shown by taking $\A$ to be the set of squares or more generally the set of integer values taken by any rational quadratic, that is to say quadratic with rational coefficients. 

It has been speculated, most particularly by Helfgott and Venkatesh \cite[pp 232--233]{helfgott-venkatesh} and by Walsh \cite{walsh}, that quadratics provide the only examples of sets for which the large sieve bound is essentially sharp. See also \cite[Problem 7.4]{croot-lev}. One might call problems of this type the ``inverse large sieve problem''. Unfortunately, we have not been able to prove any statement of this kind, and our aims here are more modest. 

Suppose that $\A \md{p} \subset S_p$ for all sufficiently large primes $p$. Our first set of results consists of improvements to the large sieve bound when $S_p$ looks very much unlike a quadratic set modulo $p$, for example by having some additive structure.

\begin{theorem}\label{thm1.4}
Suppose that for each prime $p \leq X^{1/2}$ one has a set $S_p \subset \Z/p\Z$ of size $(p+1)/2$. Suppose there is some $\delta > 0$ such that, for each $p$, $S_p$ has at least $(\frac{1}{16} + \delta) p^3$ additive quadruples, that is to say quadruples $(s_1,s_2,s_3,s_4)$ with $s_1 + s_2 = s_3 + s_4$. Suppose that $\A \mdlem{p} \subset S_p$ for all $p$. Then $|\A[X]| \ll X^{1/2 - c\delta^2}$, where $c > 0$ is an absolute constant.\end{theorem}

The condition of having at least $(\frac{1}{16} + \delta) p^3$ additive quadruples will {\em not} be satisfied by a generic (e.g. randomly selected) set $S_p \subset \Z/p\Z$ of size $(p+1)/2$, which will have $(\frac{1}{16} + o(1)) p^3$ such quadruples for large $p$. But it is a rather general condition corresponding to $S_p$ being additively structured, and certainly we are not aware of any previous improvements to the large sieve bound under comparably general conditions.

An extreme case of the preceding theorem is that in which $S_p$ is in fact an interval. Here a simple calculation, reproduced later, shows that Theorem \ref{thm1.4} is applicable with the choice $\delta=1/48$, but we can do rather better.

\begin{theorem}\label{interval-sieve}
Suppose that $\A$ is a set of integers and that, for each prime $p$, the set $\A \mdlem{p}$ lies in some interval $I_p$. Let $\eps > 0$ be arbitrary. Then
\begin{enumerate} 
\item If $|I_p| \leq (1 - \eps) p$ for at least a proportion $\eps$ of the primes in each dyadic interval $[Z,2Z]$ then $|\A[X]| \ll_{\eps} (\log \log X)^{C\log(1/\eps)}$, where $C > 0$ is some absolute constant;
\item If $|I_p| \leq \frac{p}{2}$ for all primes then $|\A[X]| \ll (\log \log X)^{\gamma + o(1)}$, where $\gamma = \frac{\log 18}{\log(3/2)} \approx 7.129$;
\item If $|I_p| = [\alpha p, \beta p]$ for all primes $p$ and for fixed $0 \leq \alpha < \beta < 1$ \textup{(}not depending on $p$\textup{)} then $|\A| = O_{\beta - \alpha}(1)$;
\item There are examples of infinite sets $\A$ with $|I_p| \leq (\frac{1}{2} + \eps) p$ for all $p$.
\end{enumerate}
\end{theorem}

This improves on the results of an unpublished preprint~\cite{green} by the first author, in which it was shown that one has $|\A[X]| \ll X^{1/3+o(1)}$ under the condition (ii). 

Theorem \ref{thm1.4} also covers the case in which $S_p$ is an arithmetic progression of length $\frac{1}{2}(p+1)$, where the common difference of this arithmetic progression may depend on $p$. Here again one could apply Theorem \ref{thm1.4} with the choice $\delta=1/48$, but we can also handle this situation with a less restrictive condition on the size of $S_p$. 

\begin{theorem}\label{prog-thm}
Let $\eps > 0$. Suppose that $\A$ is a set of integers and that, for each prime $p \leq X^{1/2}$, the set $\A \mdlem{p}$ lies in some arithmetic progression $S_p$ of length $(1-\eps)p$. Then $|\A[X]| \ll_{\eps} X^{1/2 - \eps'}$, where $\eps' > 0$ depends on $\eps$ only.
\end{theorem}

We are not aware of any previous results improving the large sieve bound $X^{1/2}$ when the $S_p$ are arbitrary arithmetic progressions, even for $\eps = 1/2$.

\vspace{11pt}
After proving the foregoing results, we turn to the ``robustness'' of the inverse large sieve problem. The aim of these results is to show that if $|\A \md{p}| \leq \frac{1}{2}(p+1)$ (or if similar conditions hold), if $|\A[X]| \approx X^{1/2}$, and if $\A$ is even vaguely close to quadratic in structure, then it must in fact approximate a quadratic very closely. Our proof methods here lead to some complicated dependencies between parameters, so we do not state and prove the most general result possible, settling instead for a couple of statements that have relatively clean formulations.

The first and main one concerns finite sets. Here, and henceforth in the paper, we say that a rational quadratic $\psi$ has \emph{height at most $H$} if it can be written as $\psi(x) = \frac{1}{d}(ax^2 + bx + c)$ with $a,b,c,d \in \Z$ and $\max(|a|, |b|, |c|, |d|) \leq H$.
\begin{theorem}\label{stability-thm}
Let $X_0 \in \N$, and let $\eps > 0$. Let $X \in \N$ be sufficiently large in terms of $X_0$ and $\eps$, and suppose that $H \leq X^{1/8}$. Suppose that $A,B \subset [X]$ and that $|A \md{p}| + |B \md{p}| \leq p+1$ for all $p \in [X_0, X^{1/4}]$. Then, for some absolute constant $c > 0$, one of the following holds:
\begin{enumerate}
\item \textup{(Better than large sieve)} Either $|A \cap [X^{1/2}]|$ or $|B \cap [X^{1/2}]|$ is $\leq X^{1/4 - c\eps^3}$;
\item \textup{(Behaviour with quadratics)} Given any two rational quadratics $\psi_A,\psi_B$ of height at most $H$, either $|A \setminus \psi_{A}(\Q)|$ and $|B\setminus \psi_{B}(\Q)| \leq HX^{1/2 - c}$, or else at least one of $|A \cap \psi_{A}(\Q)|$ and $|B \cap \psi_{B}(\Q)|$ is bounded above by $HX^{1/4 + \eps}$.
\end{enumerate}
\end{theorem}

We expect that if the large sieve bound is close to sharp for $A$ and $B$, then there must exist rational quadratics of ``small'' height containing ``many'' points of $A$ and $B$. Together with Theorem \ref{stability-thm}, this provides some motivation for making the following conjecture of the form ``almost equality in the large sieve implies quadratic structure''.

\begin{conjecture}\label{stability-conj}
Let $X_0 \in \N$, and let $\rho > 0$. Let $X \in \N$ be sufficiently large in terms of $X_0$ and $\rho$. Suppose that $A,B \subset [X]$ and that $|A \md{p}| + |B \md{p}| \leq p+1$ for all $p \in [X_0, X^{1/4}]$. Then there exists a constant $c = c(\rho) > 0$ such that one of the following holds:
\begin{enumerate}
\item \textup{(Better than large sieve)} Either $|A \cap [X^{1/2}]|$ or $|B \cap [X^{1/2}]|$ is $\leq X^{1/4 - c}$;
\item \textup{(Quadratic structure)} There are two rational quadratics $\psi_A,\psi_B$ of height at most $X^{\rho}$ such that $|A \setminus \psi_{A}(\Q)|$ and $|B\setminus \psi_{B}(\Q)| \leq X^{1/2 -c}$.
\end{enumerate}
\end{conjecture}

The contents of Theorem \ref{stability-thm} and of Conjecture \ref{stability-conj} are perhaps a little hard to understand on account of the parameters $H, X, \rho$ and $\eps$. As a corollary we establish the following more elegant statement involving infinite sets. 

\begin{theorem}\label{stability-thm-symmetric}
Suppose that $\A$ is a set of positive integers and that $|\A \md{p}| \leq \frac{1}{2}(p+1)$ for all sufficiently large primes $p$. Then one of the following options holds:
\begin{enumerate}
\item \textup{(Quadratic structure)} There is a rational quadratic $\psi$ such that all except finitely many elements of $\A$ are contained in $\psi(\Q)$;
\item \textup{(Better than large sieve)} For each integer $k$ there are arbitrarily large values of $X$ such that $|\A[X]| < \frac{X^{1/2}}{\log^k X}$;
\item \textup{(Far from quadratic structure)} Given any rational quadratic $\psi$, for all $X$ we have $|\A[X] \cap \psi(\Q)| \leq X^{1/4 + o_{\psi}(1)}$.
\end{enumerate}
\end{theorem}

We conjecture that option (iii) is redundant. This is another conjecture of inverse large sieve type, rather cleaner than Conjecture \ref{stability-conj}.

\begin{conjecture}\label{stability-conj-symmetric}
Suppose that $\A$ is a set of positive integers and that $|\A \md{p}| \leq \frac{1}{2}(p+1)$ for all sufficiently large primes $p$. Then one of the following options holds:
\begin{enumerate}
\item \textup{(Quadratic structure)} There is a rational quadratic $\psi$ such that all except finitely many elements of $\A$ are contained in $\psi(\Q)$;
\item \textup{(Better than large sieve)} For each integer $k$ there are arbitrarily large values of $X$ such that $|\A[X]| < \frac{X^{1/2}}{\log^k X}$. In particular, $\lim\inf_{X \rightarrow \infty} X^{-1/2} |\A[X]| = 0$.
\end{enumerate}
\end{conjecture}

We remark that some very simple properties of rational quadratics are laid down in Appendix \ref{rat-quad-app}. In particular we draw attention to the fact that given a rational quadratic $\psi$ there are further rational quadratics $\psi_1,\psi_2$ such that $\psi_1(\Z) \subset \psi(\Q) \cap \Z \subset \psi_2(\Z)$. In particular, $|\psi(\Q) \cap [X]| \ll_{\psi} X^{1/2}$. 

\vspace{11pt}
Our final task in the paper is to show, elaborating on ideas of Elsholtz \cite{elsholtz}, that Conjecture \ref{stability-conj} would resolve the currently unsolved ``inverse Goldbach problem'' of Ostmann \cite[p. 13]{ostmann} (and see also \cite[p. 62]{erdos}). This asks whether the set of primes can be written as a sumset $\A + \B$ with $|\A|, |\B| \geq 2$, except for finitely many mistakes. Evidently the answer should be that it cannot be so written.

\begin{theorem}\label{ils-ostmann}
Assume Conjecture \ref{stability-conj}. Let $\A, \B$ be two sets of positive integers, with $|\A|, |\B| \geq 2$, such that $\A + \B$ contains all sufficiently large primes. Then $\A + \B$ also contains infinitely many composite numbers. 
\end{theorem}

We remark that much stronger statements that would imply this should be true, but we do not know how to prove them. For example, it is reasonable to make the following conjecture.

\begin{conjecture}
Let $\delta > 0$. Then if $X$ is sufficiently large in terms of $\delta$, the following is true. Let $A, B \subset [X]$ be any sets with $|A|, |B| \geq X^{\delta}$. Then $A + B$ contains a composite number.
\end{conjecture}
We do not know how to prove this for any $\delta \leq \frac{1}{2}$. If one had it for any $\delta < \frac{1}{2}$, the inverse Goldbach problem would follow.

\vspace{11pt}
The proofs of the above theorems are rather diverse and use the large sieve, Gallagher's ``larger sieve'', and several other tools from harmonic analysis and analytic number theory. The proofs of Theorems \ref{thm1.4} and \ref{prog-thm}, which involve lifting the additive structure of the $S_p$ to additive structure on $\A[X]$, also involve some ideas of additive combinatorial flavour, although we do not need to import many results from additive combinatorics to prove them. With very few exceptions (for example, Lemma \ref{progr-energy} depends on standard Fourier arguments given in detail in Lemma \ref{large-fourier}), Sections 3,4,5,6 and 7 may be read independently of one another.

The situation considered in the majority of this paper, in which $\A \md{p}$ is small for \emph{every} prime $p$ (or at least for every prime $p \leq X^{1/2}$), may seem rather restrictive. It would be possible to adapt our arguments and prove many of our theorems under weaker conditions, and we leave this to the reader who has need of such results. However, it seems possible that any set $\A$ for which $|\A \md{p}| \leq (1 - c)p$ for a decent proportion of primes $p$ and for which $|\A[X]| \geq X^c$ for infinitely many $X$ has at least some ``algebraic structure''.  Moreover such statements may well be true in finitary settings, in which $\A$ is restricted to some finite interval $[X]$ and $p$ is only required to range over some (potentially quite small) subinterval of $[X]$. Unfortunately none of our methods come close to establishing such strong results. \vspace{11pt}

\textsc{Acknowledgements.} The authors are very grateful to Jean Bourgain for allowing us to use his ideas in Section \ref{interval-sieve-sec}. The first-named author is supported by ERC Starting Grant 279438 \emph{Approximate Algebraic Structure and Applications}. He is very grateful to the University of Auckland for providing excellent working conditions on a sabbatical during which this work was completed. The second-named author was supported by a Doctoral Prize from the EPSRC when this work was started; by a postdoctoral fellowship from the Centre de Recherches Math\'ematiques, Montr\'eal; and by a research fellowship at Jesus College, Cambridge when the work was completed.

\section{The large sieve and the larger sieve}

\textsc{The large sieve}. Let us begin by briefly recalling a statement of the large sieve bound. The following may be found in Montgomery \cite{montgomery-early}.

\begin{proposition}\label{classicls}
Let $\mathscr{A}$ be a set of positive integers with the property that $\mathscr{A} \mdlem{p} \subset S_p$ for each prime $p$. Then for any $Q,X$ we have the bound
$$ |\mathscr{A}[X]| \leq \frac{X + Q^{2}}{\sum_{q \leq Q} \mu^{2}(q) \prod_{p \mid q} \frac{|S_p^c|}{|S_p|}} \leq \frac{X + Q^2}{\sum_{p \leq Q} \frac{|S_p^c|}{p}}. $$
where $\mu(q)$ denotes the M\"{o}bius function.
\end{proposition}

The second bound is a little crude but has the virtue of being simple: we will use it later on. In the particular case that $|S_p| \leq \frac{1}{2}(p+1)$ for all $p$, discussed in the introduction, the first bound implies upon setting $Q := X^{1/2}$ that
$$ |\mathscr{A}[X]| \leq \frac{2X}{\sum_{q \leq X^{1/2}} \mu^{2}(q) \prod_{p \mid q} \frac{p-1}{p+1}} \ll X^{1/2}, $$ as we claimed.

The large sieve may also be profitably applied to ``small sieve'' situations in which $|S_p| = p - O(1)$ (as opposed to ``large sieve'' situations in which $p - |S_p|$ is large). We will need one such result later on, in \S \ref{stab-sec}. 

\begin{lemma}\label{small-from-large}
Suppose that $\mathscr{B} \subset \Z$ is a set with the property that $\mathscr{B} \mdlem{p}$ misses $w(p)$ residue classes, for every prime $p$. Suppose that the function $w$ has average value $k$ in the \textup{(}fairly weak\textup{)} sense that $\frac{1}{Z}\sum_{Z \leq p \leq 2Z} w(p) \log p = k + O(\frac{1}{\log^2 Z})$ for all $Z$. Then $|\mathscr{B} \cap [-N, N]| \ll N (\log N)^{-k}$ for all $N$.
\end{lemma}
\begin{proof}
In view of Proposition \ref{classicls}, it would suffice to know that
$$ \sum_{q \leq N^{1/2}} \mu^2(q) \prod_{p | q} \frac{w(p)}{p - w(p)} \gg \log^k N$$
for all large $N$. If we define a multiplicative function $g(n)$, supported on squarefree integers, by $g(p) := w(p)/p$, then it would obviously suffice to know that $\sum_{q \leq N^{1/2}} g(q) \gg \log^k N$ for all large $N$. However there is an extensive theory, dating back to Hal\'{a}sz, Wirsing and others, that gives asymptotics and bounds for sums of multiplicative functions. For example, partial summation and the assumption that $\sum_{Z \leq p \leq 2Z} w(p) \log p = k Z + O(\frac{Z}{\log^2 Z})$ show that $g$ satisfies the conditions of Theorem A.5 of Friedlander and Iwaniec~\cite{opera-cribo}, and therefore
$$ \sum_{q \leq N} g(q) \sim c_{g} \log^k N \;\;\; \text{as} \; N \rightarrow \infty , $$
for a certain constant $c_{g} > 0$.
\end{proof}

\textsc{The larger sieve}. The ``larger sieve'' was introduced by Gallagher \cite{gallagher}. A pleasant discussion of it may be found in chapter 9.7 of Friedlander and Iwaniec \cite{opera-cribo}. We will apply it several times in the paper, and we formulate a version suitable for those applications.

\begin{theorem}\label{larger-sieve}
Let $0 < \delta \leq 1$, and let $Q > 1$. Let $\mathscr{P}$ be a set of primes. For each prime $p \in \mathscr{P}$, suppose that one is given a set $S_p \subset \Z/p\Z$, and write $\sigma_p := |S_p|/p$. Suppose that there is some set $\A'_p \subset \A$, $|\A'_p[X]| \geq \delta |\A[X]|$, such that $\A'_p \md{p} \subset S_p$. Then 
\[ |\A[X]| \ll \frac{Q}{\delta^2 \sum_{p \in \mathscr{P}, p \leq Q} \frac{\log p}{\sigma_p p} - \log  X},\] provided that the denominator is positive.
\end{theorem}
\emph{Remark.} In this paper we will always have $\delta$ at least some absolute constant, not depending on $X$, and very often we will have $\delta \approx 1$.
\begin{proof}
We examine the expression
\[ I := \sum_{\substack{p \in \mathscr{P} \\ p \leq Q}} \sum_{\substack{x,y \in \A[X] \\ x \neq y}} 1_{p | x - y} \log p.\]
On the one hand we have
\[ \sum_{p \in \mathscr{P}} 1_{p | n} \log p \leq \sum_{p} 1_{p | n} \log p \leq \log n,\] and therefore 
\[ I \leq |\A[X]|^2 \log X.\]
On the other hand, writing $\A(a,p;X)$ for the number of $x \in \A[X]$ with $x \equiv a \md{p}$, we have
\[ I = \sum_{\substack{p \in \mathscr{P}\\ p \leq Q}} \sum_{a \mdsub{p}} |\A(a,p;X)|^2 \log p - |\A[X]| \sum_{\substack{p \in \mathscr{P}\\ p \leq Q}} \log p.\]
Comparing these facts yields
\begin{equation}\label{to-use} \sum_{\substack{p \in \mathscr{P} \\ p \leq Q}} \sum_{a \mdsub{p}} |\A(a,p;X)|^2 \log p \leq |\A[X]|^2 \log X + O(|\A[X]|Q),\end{equation}
and so of course, since $\A'_p \subset \A$,
\[ \sum_{\substack{p \in \mathscr{P} \\ p \leq Q}} \sum_{a \mdsub{p}} |\A'_p(a,p;X)|^2 \log p \leq  |\A[X]|^2 \log X + O(|\A[X]|Q).\]
However by the Cauchy--Schwarz inequality and the fact that $\A'_p \md{p} \subset S_p$ we have
\[ \sum_{a \mdsub{p}} |\A'_p(a,p;X)|^2 \geq \frac{|\A'_p[X]|^2}{\sigma_p p} \geq \delta^2 \frac{|\A[X]|^2}{\sigma_p p}.\]
Summing over $p$ and rearranging, we obtain the result.
\end{proof}
The larger sieve bound can be a little hard to get a feel for, so we give an example. Suppose that $\mathscr{P}$ consists of all primes and that $\sigma_p = \alpha$ for all $p$. Take $\delta = 1$. Then, since $\sum_{p \leq Q} \frac{\log p}{p} = \log Q + O(1)$, the larger sieve bound is essentially 
\[ |\A[X]| \ll \frac{Q}{\frac{1}{\alpha} \log Q - \log X}.\]
Taking $Q$ a little larger than $X^{\alpha}$, we obtain the bound $|\A[X]| \ll X^{\alpha + o(1)}$. This beats an application of the large sieve when $\alpha < \frac{1}{2}$, that is to say when we are sieving out a majority of residue classes (hence the terminology ``larger sieve''). However in the type of problems we are generally considering in this paper, where $\alpha = \frac{1}{2}$, we only recover the bound $|\A[X]| \ll X^{1/2 + o(1)}$, as of course we must since nothing has been done to exclude the example where $\A$ is a set of values of a quadratic.

In actual fact one of our three applications of the larger sieve (in the proof of Theorem \ref{thm1.4}) requires an inspection of the above proof, rather than an application of the result itself. This is the observation that when $\sigma_p \approx \frac{1}{2}$ the larger sieve \emph{does} beat the bound $|\A[X]| \ll X^{1/2 + o(1)}$ unless $\A$ satisfies a certain ``uniform fibres'' condition. Recall that if $\A$ is a set of integers then $\A(a,p;X)$ is the number of $x \in \A[X]$ with $x \equiv a \md{p}$.

\begin{lemma}\label{lem1.1}
Let $\kappa > 0$ and $\eta > 0$ be small parameters. Suppose that $\A$ is a set of integers occupying at most $(p+1)/2$ residue classes modulo $p$ for all $p$. Say that $\A[X]$ has {\em $\eta$-uniform fibres above $p$} if 
\[ \sum_{a \mdsublem{p}} |\A(a, p;X)|^2 \leq (2 + \eta) |\A[X]|^2/p.\]  Let $\mathscr{P}_{\unif}$ be the set of primes above which $\A[X]$ has $\eta$-uniform fibres. Then either $|\A[X]| \leq X^{1/2 - \kappa}$ or else ``most'' fibres are $\eta$-uniform in the sense that
 \[ \sum_{\substack{p \leq X^{1/2}, \\ p \notin \mathscr{P}_{\unif}}} \frac{\log p}{p} \leq \frac{3\kappa \log X + O(1)}{\eta}.\]
\end{lemma}
\begin{proof}
Let $\mathscr{P}$ be the set of all primes, and let $Q  := X^{1/2 - \kappa}$. 
We proceed as in the proof of the larger sieve until \eqref{to-use}, which was the inequality
\[ \sum_{p \leq Q} \sum_{a \mdsub{p}} |\A(a,p;X)|^2 \log p \leq |\A[X]|^2 \log X + O(|\A[X]|Q).\]
Now by the Cauchy--Schwarz inequality we have \[ \sum_{a \mdsub{p}} |\A(a,p;X)|^2 \geq 2|\A[X]|^2/(p+1)\] for all $p$. Using this and the estimate $\sum_{p \leq Q} \log p/p = \log Q + O(1)$, we see that the left-hand side of \eqref{to-use} is at least
\[ 2|\A[X]|^2 (\log Q + O(1)) + \eta |\A[X]|^2 \sum_{\substack{p \leq Q \\ p \notin \mathscr{P}_{\unif}}} \frac{\log p}{p}.\]
Therefore
\[ \eta \sum_{\substack{p \leq Q \\ p \notin \mathscr{P}_{\unif}}} \frac{\log p}{p} \leq \log X - 2 \log Q + O(1) + O(\frac{Q}{|\A[X]|}).\] If $|\A[X]| \leq X^{1/2 - \kappa}$ then we are done; otherwise, the term $O(Q/|\A[X]|)$ may be absorbed into the $O(1)$ term and, after a little rearrangement, we obtain 
\[ \sum_{\substack{p \leq Q \\ p \notin \mathscr{P}_{\unif}}} \frac{\log p}{p} \leq \frac{2 \kappa \log X + O(1)}{\eta}.\]
Since 
\[ \sum_{Q \leq p \leq X^{1/2}} \frac{\log p}{p} = \kappa \log X + O(1),\] the claimed bound follows.
\end{proof}

\section{Sieving by additively structured sets}

Our aim in this section is to establish Theorem \ref{thm1.4}.

Let $A$ be a finite set of integers. As is standard, we write $E(A,A)$ for the {\em additive energy} of $A$, that is to say the number of quadruples $(a_1,a_2,a_3,a_4) \in A^4$ with $a_1 + a_2 = a_3 + a_4$. If $p$ is a prime, write $E_p(A, A)$ for the number of quadruples with $a_1 + a_2 \equiv a_3 + a_4 \md{p}$. It is easy to see that $E_p(A,A) \geq |A|^4/p$. In situations where this inequality is not tight, we can get a lower bound for the additive energy $E(A,A)$. To do this we will use the {\em analytic large sieve inequality}, which is something like an approximate version of Bessel's inequality (and which leads, in a non-obvious way, to the large sieve bound that we stated as Proposition \ref{classicls}). We cite the following version, which is best possible in various aspects, from Chapter 9.1 of Friedlander and Iwaniec \cite{opera-cribo}.

\begin{proposition}
Let $0 < \delta \leq 1/2$, and suppose that $\theta_{1}, \dots , \theta_{R} \in \R/\Z$ form a $\delta$-spaced set of points, in the sense that $\Vert \theta_{r} - \theta_{s} \Vert \geq \delta$ for all $ r \neq s$ where $\Vert \cdot \Vert$ denotes distance to the nearest integer. Suppose that $(a(x))_{M < x \leq M+X}$ are any complex numbers, where $X$ is a positive integer. Then
\[ \sum_{r=1}^{R} \bigg|\sum_{M < x \leq M+X} a(x) e(\theta_{r}x) \bigg|^{2} \leq (X - 1 + \delta^{-1}) \sum_{M < x \leq M+X} |a(x)|^{2}, \]
where as usual $e(\theta) := \exp\{2\pi i \theta\}$. 
\end{proposition}

Using the analytic large sieve inequality, we shall prove the following lemma.
\begin{lemma}[Lifting additive energy]\label{lift-lem} Suppose that $A \subset [X]$. Then we have
\[ \sum_{p \leq X^{1/2}} p \big( E_p(A,A) - \frac{|A|^4}{p} \big) \leq 3 X E(A,A).\]
\end{lemma}
\begin{proof} Write $r(x)$ for the number of representations of $x$ as $a_1 + a_2$ with $a_1,a_2 \in A$. Then 
\[ E(A,A) = \sum_{x \leq 2X} r(x)^2,\]
whilst
\[ E_p(A, A) = \sum_{\substack{x,x' \leq 2X \\ x \equiv x' \mdsub{p}}} r(x) r(x') = \frac{1}{p}\sum_{a \mdsub{p}} |\sum_{x \leq 2X} r(x) e(ax/p)|^2.\]
It follows that
\begin{align*} \sum_{p \leq X^{1/2}} p E_p(A,A) & = \sum_{p \leq X^{1/2}} \sum_{a \mdsub{p}} |\sum_{x \leq 2X} r(x) e(ax/p)|^2  \\ & = \sum_{p \leq X^{1/2}} \sum_{\substack{a \mdsub{p} \\ a \neq 0}} |\sum_{x \leq 2X} r(x) e(ax/p)|^2 + \sum_{p \leq X^{1/2}} p \frac{|A|^4}{p}.\end{align*}
Now the fractions $a/p$ are $1/X$-spaced, as $a, p$ range over all pairs with $p \leq X^{1/2}$ prime and $1 \leq a \leq p-1$. By the analytic form of the large sieve it follows that
\[ \sum_{p \leq X^{1/2}} \sum_{a \mdsub{p} a \neq 0} |\sum_{x \leq 2X} r(x) e(ax/p)|^2 \leq 3 X \sum_{x \leq 2X}  r(x)^2.\]
Putting all these facts together gives the result.\end{proof}

\begin{corollary}\label{cor1.3}
Let $\eta,\delta > 0$ be small real numbers with $\eta \leq \delta^2$. Suppose that $A \subset [X]$ is a set. Let $\mathscr{P}$ be a set of primes satisfying $36\delta^{-2} \leq p \leq X^{1/2}$, and suppose that the following are true whenever $p \in \mathscr{P}$:
\begin{enumerate}
\item $A \mdlem{p}$ lies in a set $S_p$ of cardinality at most $\frac{1}{2}(p+1)$;
\item $S_p$ has at least $(\frac{1}{16} + \delta) p^3$ additive quadruples;
\item $A$ has $\eta$-uniform fibres mod $p$, in the sense that $\sum_{a \mdsublem{p}} |A(a; p)|^2 \leq (2 + \eta) |A|^2/p$,
where $A(a;p)$ is the number of $x \in A$ with $x \equiv a \mdlem{p}$.
\end{enumerate}
Then $E(A,A) \geq \frac{\delta |A|^4}{3X}|\mathscr{P}|$.
\end{corollary}
\begin{proof}
Suppose that $p \in \mathscr{P}$. We will obtain a lower bound for $E_p(A,A)$ which beats the trivial bound of $E_p(A,A) \geq |A|^4/p$.  The corollary will then follow quickly from Lemma \ref{lift-lem}. First of all we apply the variance identity
\[ \sum_{m = 1}^M (t_m - \frac{1}{M}\sum_{i = 1}^M t_i)^2 = \sum_{i=1}^M t_i^2 - \frac{1}{M} (\sum_{i=1}^M t_i)^2\]
with $M := |S_p|$ and the $t_i$ being the $A(a;p)$ with $a \in S_p$. This and the uniform fibres assumption yields
\[ \sum_{a \in S_p} \big(|A(a;p)| - \frac{|A|}{|S_p|}\big)^2 \leq \frac{(2 + \eta) |A|^2}{p} - \frac{2|A|^2}{p+1} \leq \frac{|A|^2}{p}\big(\eta + \frac{2}{p}\big).\]
Write $f : \Z/p\Z \rightarrow \R$ for the function $f(a) := |A(a;p)|$, and $g : \Z/p\Z \rightarrow \R$ for the function which is $|A|/|S_p|$ on $S_p$ and zero elsewhere. We have shown\footnote{The normalisations here are the ones standard in additive combinatorics. Write $\Vert F \Vert_2 = (\frac{1}{p}\sum_{x \in \Z/p\Z} |F(x)|^2)^{1/2}$, but $\Vert \hat{F} \Vert_2 = (\sum_r |\hat{F}(r)|^2)^{1/2}$ on the Fourier side. Then we have the Parseval identity $\Vert F \Vert_2 = \Vert \hat{F} \Vert_2$ and Young's inequality $\Vert \hat{F} \Vert_{\infty} \leq \Vert F \Vert_1$, both of which are used here.} that
\[ \Vert f - g \Vert_2 \leq \frac{|A|}{p} \sqrt{\eta + \frac{2}{p}}.\]
Note also that, by the Cauchy--Schwarz inequality,
\[ \Vert f - g \Vert_1 \leq \Vert f - g \Vert_2 \leq \frac{|A|}{p} \sqrt{\eta + \frac{2}{p}} ,\]
and hence
\[ \Vert \hat{f} - \hat{g} \Vert_4^4 \leq \Vert \hat{f} - \hat{g} \Vert_{\infty}^{2} \Vert \hat{f} - \hat{g} \Vert_2^{2} \leq \Vert f - g \Vert_1^2 \Vert f - g \Vert_2^2 \leq \frac{|A|^4}{p^4}(\eta + \frac{2}{p})^{2},\] which of course implies that
\[ \Vert \hat{f} - \hat{g} \Vert_4 \leq \frac{|A|}{p} (\eta^{1/2} + \sqrt{\frac{2}{p}}),\]
where $\hat{f}(r) := \frac{1}{p}\sum_{a \in \Z/p\Z} f(a)e(-ar/p)$, similarly for $\hat{g}$. It follows that
\[ \Vert \hat{f} \Vert_4 \geq \Vert \hat{g} \Vert_4 - \frac{\eta^{1/2}|A|}{p} - \frac{\sqrt{2}|A|}{p^{3/2}}.\]

Note, however, that
\[ E_p(A,A) = \sum_{a_1 + a_2 = a_3 + a_4} f(a_1)f(a_2) f(a_3) f(a_4) = p^3 \Vert \hat{f} \Vert_4^4,\]
whilst
\[ \Vert \hat{g} \Vert_4^4 = \frac{|A|^4}{|S_p|^4} \frac{E_p(S_p, S_p)}{p^{3}} \geq \frac{|A|^4}{16|S_p|^4} (1 + 16\delta)\] and so
\[ \Vert \hat{g} \Vert_4 \geq \frac{|A|}{2|S_p|}(1 + 2\delta) \geq \frac{|A|}{p+1}(1 + 2\delta).\]
Putting these facts together, and remembering that $\eta \leq \delta^{2}$ and $p \geq 36\delta^{-2}$, yields
\[ \Vert \hat{f} \Vert_4 \geq \frac{|A|}{p}(1 + 2\delta - \eta^{1/2} - 3/p^{1/2}) \geq \frac{|A|}{p}(1 + \delta/2)\] and so $E_p(A,A) \geq \frac{|A|^4}{p}(1 + \delta)$ whenever $p \in \mathscr{P}$.
The result now follows immediately from Lemma \ref{lift-lem}.
\end{proof}

\begin{corollary}\label{cor3.3}
Let $\kappa > 0$ be a small parameter. Suppose that $A \subset [X]$ and that, for every prime $p \leq X^{1/2}$, the set $A \mdlem{p}$ lies in a set $S_p$ of cardinality at most $\frac{1}{2}(p+1)$ and with at least $(\frac{1}{16} + \delta) p^3$ additive quadruples. Then either $|A| \ll X^{1/2 - \kappa}$ or else $E(A,A) \geq \delta X^{-9 \kappa/\delta^2} |A|^3$.
\end{corollary}
\begin{proof}
Suppose that $|A| \geq X^{1/2 - \kappa}$ and that $\kappa \log X$ is large enough. (If $\kappa \log X$ is small then $|A| \ll X^{1/2} \ll X^{1/2 - \kappa}$ by the usual large sieve bound.) Set $\eta := \delta^2$. By Lemma \ref{lem1.1} we either have $|A| \leq X^{1/2 - \kappa}$ or else $A$ has $\eta$-uniform fibres on a set $\mathscr{P} \subset  [X^{1/2}]$ of primes satisfying
\[ \sum_{\substack{p \leq X^{1/2} \\ p \notin \mathscr{P}}} \frac{\log p}{p} \leq \frac{4\kappa \log X}{\eta}.\] This implies that $|\mathscr{P}| \geq X^{1/2 - 8\kappa/\eta}$, and so by Corollary \ref{cor1.3} and the fact that $|A| \geq X^{1/2 - \kappa}$ we have $E(A,A) \geq \delta X^{-9\kappa/\eta} |A|^3$, which gives the claimed bound.
\end{proof}

The main task for the rest of this section will be to prove the following. 

\begin{proposition}[Differenced larger sieve]\label{diff-largersieve}
Let $X$ be large, and let $A \subset [X]$ be a set with the property that $A \mdlem{p}$ lies in a set $S_p$ of size at most $\frac{1}{2}(p+1)$ for all primes $p \leq X^{1/2}$. Suppose that $E(A,A) \geq |A|^3/K$. Then $|A| \leq K X^{1/2 - c_0}$, where $c_0 > 0$ is an absolute constant.
\end{proposition}

Let us pause to see how this and Corollary \ref{cor3.3} combine to establish Theorem \ref{thm1.4}.

\begin{proof}[Proof of Theorem \ref{thm1.4} given Proposition \ref{diff-largersieve}]
Let $\kappa > 0$ be a parameter to be specified shortly. Suppose that $A \subset [X]$, and that $A \mdlem{p} \subset S_p$ for all primes $p$. Suppose furthermore that $|S_p| = \frac{1}{2}(p+1)$ and that $S_p$ has at least $(\frac{1}{16} + \delta) p^3$ additive quadruples for all $p$.  By Corollary \ref{cor3.3} we see that either $|A| \ll X^{1/2 - \kappa}$, or else $E(A,A) \geq \delta X^{-9\kappa/\delta^2} |A|^3$. In this second case it follows from Proposition \ref{diff-largersieve} that $|A| \ll_{\delta}  X^{\frac{1}{2} + 9 \kappa/\delta^2 - c_0}$.
Choosing $\kappa := c_0\delta^2/10$ gives the result.
\end{proof}

It remains to prove Proposition \ref{diff-largersieve}. As the reader will soon see, the proof might be thought of as a ``differenced larger sieve'' argument, in which the larger sieve is not applied to $A$ directly, but rather to intersections of shifted copies of $A$ (as in Lemma \ref{lem2.7}) and to a set $H$ of pairwise differences of elements of $A$ (as in Lemma \ref{lem2.8}). The assumption that $A$ has large additive energy allows one to recover bounds on $A$ from that information (as in Lemma \ref{lem3.5}).

\emph{Remark.} It is possible to prove Proposition \ref{diff-largersieve} with a quite respectable value of the constant $c_{0}$. Unfortunately the quality of the final bound in Theorem \ref{thm1.4} is not really determined by the value of $c_{0}$, but by the much poorer bounds that we achieved when trying to force the set $A$ to have uniform fibres mod $p$. We believe that by reworking Corollary \ref{cor1.3} a little one could prove Theorem \ref{thm1.4} with an improved bound $|\A[X]| \ll X^{1/2 - c\sqrt{\delta}}$, but this is presumably very far from optimal.

\begin{proof}[Proof of Proposition \ref{diff-largersieve}]
The argument is a little involved, so we begin with a sketch. Suppose that $E(A,A) \approx |A|^3$. Then it is not hard to show that $|A \cap (A + h)| \approx |A|$ for $h \in H$, where $|H| \approx |A|$. Modulo $p$, the set $A \cap (A + h)$ is contained in $S_p \cap (S_p + h)$. If, for some $h \in H$, we have $|S_p \cap (S_p + h)| < (\frac{1}{2} - c)p$ then an application of the larger sieve implies that $|A \cap (A + h)| < X^{1/2 - c'}$, and hence $|A| \lessapprox X^{1/2 - c'}$. The alternative is that $|S_p \cap (S_p + h)| \approx \frac{1}{2}p$ for many $p$, for all $h \in H$. Using this we can show that there is \emph{some} $p$ for which $|S_p \cap (S_p + h)| \approx \frac{1}{2}p$ for many $h$. By a result of Pollard, there is no such set $S_p$.

Let us turn now to the details, formulating a number of lemmas which correspond to the above heuristic discussion. From now on, the assumptions are as in Proposition \ref{diff-largersieve}.

\begin{lemma}\label{lem3.5} Then there is a set $H \subset [-X, X]$, $|H| \geq |A|/2K$ such that $|A \cap (A + h)| \geq |A|/2K$ for all $h \in H$.
\end{lemma}
\begin{proof}
This is completely standard additive combinatorics and is a consequence, for example, of the inequalities in \cite[\S 2.6]{tao-vu}. It is no trouble to give a self-contained proof: note that $E(A,A) = \sum_x |A \cap (A + x)|^2$ and that we have the trivial bound $|A \cap (A + x)| \leq |A|$ for all $x$. If $H$ is the maximal set with the stated property then
\[ E(A, A)  = \sum_{x \in H} |A \cap (A + x)|^2 + \sum_{x \notin  H} |A \cap (A + x)|^2  \leq |H| |A|^2 + \frac{|A|}{2K}\sum_{x} |A \cap (A + x)|  = |H| |A|^2 + \frac{|A|^3}{2K},\]
from which the statement follows immediately.
\end{proof}

\begin{lemma}\label{lem2.7} Let $c > 0$ be a small constant. Set $Q := X^{1/2 - c/2}$, and suppose that there is some $h \in H$ such that 
\[ \sum_{p \leq Q} \frac{\log p}{p} \frac{|S_p \cap (S_p + h)|}{p} < (\frac{1}{2} - c) \log Q.\] Then $|A| \ll K X^{1/2 - c/4}$.
\end{lemma}
\begin{proof}
Note that $A \cap (A + h) \subseteq S_p \cap (S_p + h)$ for all $p$. Write $\sigma_p := |S_p \cap (S_p + h)|/p$ and apply the larger sieve, Theorem \ref{larger-sieve}, with $\delta = 1$ and $A$ replaced by $A \cap (A + h)$. We obtain the bound

\begin{equation}\label{eq77denom} |A \cap (A + h)| \ll \frac{Q}{\sum_{p \leq Q} \frac{\log p}{p\sigma_p} - \log X},\end{equation}
provided that the denominator is positive.

Our assumption is that 
\[ \sum_{p \leq Q} \frac{\sigma_p \log p}{p} \leq (\frac{1}{2} - c) \log Q.\] Since $4t + 1/t \geq 4$ for all $t > 0$, it follows that 
\[ \sum_{p \leq Q} \frac{\log p}{\sigma_p p} \geq 4\log Q + O(1) - 4(\frac{1}{2} - c)\log Q = (2 + 4c)\log Q + O(1).\] 
It is easy to check that the denominator of \eqref{eq77denom} is indeed positive, since $Q = X^{1/2 - c/2}$. We obtain the bound
\[ |A \cap (A + h)| \ll X^{1/2 - c/2 + o(1)} \ll X^{1/2 - c/4}.\]
Since $|A \cap (A + h)| \geq |A|/2K$, the lemma follows.
\end{proof}

Before stating the next lemma, let us isolate a fact which will be needed in the proof. This is basically due to Pollard.
\begin{lemma}[Pollard]\label{add-comb}
Let $\eps > 0$ be small, and let $S \subset \Z/p\Z$ be a non-empty set such that $|S| < (1-2\eps)p$. Then there are at most $4\eps |S| + 1$ values of $h \in \Z/p\Z$ such that $|S \cap (S + h)| \geq (1 - \eps)|S|$.
\end{lemma}
\begin{proof}
This follows quickly from a well-known result of Pollard \cite{pollard}. Writing $N_i$ for the number of $h$ such that $|S \cap (S + h)| \geq i$, Pollard's result in our setting implies that $N_1 + \dots + N_r \geq r(2|S|-r)$ for all $2|S| - p \leq r \leq |S|$.  Temporarily write $H$ for the set of all $h \in \Z/p\Z$ such that $|S \cap (S + h)| \geq (1 - \eps)|S|$, and also let $R := |S| - 2\lfloor \eps|S| \rfloor$ and $U := |S| - \lfloor \eps|S| \rfloor$, where $\lfloor \cdot \rfloor$ denotes the floor function. Then 
\[ N_{R+1} + \dots + N_{|S|} \geq N_{R+1} + \dots + N_{U} \geq |H|(U - R) = |H| \lfloor \eps|S| \rfloor .\]
Pollard's result tells us that
\[ N_1 + \dots + N_{R} \geq R(2|S|-R) = |S|^{2} - 4\lfloor \eps|S| \rfloor^{2} .\]
On the other hand we trivially have
\[ N_1 + \dots + N_{|S|} = |S|^2.\]
Combining all these facts leads to the result provided that $|S| \geq 1/\epsilon$.

Alternatively, if $|S| < 1/\epsilon$ then $|S \cap (S + h)| \geq (1 - \eps)|S|$ only if $S \cap (S + h) = S$, in which case $S \cap (S + nh) = S$ for every $n$. Since $S$ is a proper subset of $\Z/p\Z$, this can only happen when $h=0$. 
\end{proof}

We will also require a simple and standard averaging principle, the proof of which we include here for completeness.

\begin{lemma}\label{simple-avg}
Let $\eps, \eps'$ be real numbers with $0 < \eps \leq \eps'$. Let $X$ be a finite set, let $(\lambda(x))_{x \in X}$ be nonnegative weights,  and suppose that $f : X \rightarrow [0,1]$ is a function such that $\sum_{x \in X} \lambda(x) f(x) \geq (1 - \eps) \sum_{x \in X} \lambda(x)$.

Let $X' \subset X$ be the set of all $x \in X$ such that $f(x) \geq 1 - \eps'$. Then $\sum_{x \in X'} \lambda(x) \geq (1 - \frac{\eps}{\eps'}) \sum_{x \in X} \lambda(x)$. In particular if $\sum_{x \in X} f(x) \geq (1 - \eps) |X|$ then there are at least $(1 - \frac{\eps}{\eps'}) |X|$ values of $x$ such that $f(x) \geq 1 - \eps'$. 
\end{lemma}
\begin{proof}
We have
\[ (1 - \eps) \sum_{x \in X} \lambda(x)  \leq \sum_{x \in X} \lambda(x) f(x) = \sum_{x \in X'} \lambda(x) f(x) + \sum_{x \in X \setminus X'} \lambda(x) f(x) \leq \sum_{x \in X'} \lambda(x) + (1 - \eps') \sum_{x \in X \setminus X'} \lambda(x).\]
Rearranging this inequality gives the first result. The second one follows by taking all the weights $\lambda(x)$ to be 1.
\end{proof}

\begin{lemma}\label{lem2.8}
Let $c > 0$ be a sufficiently small absolute constant. Let $H \subset [-X,X]$ be a set of size $X^{1/2 - c/4}$, and let $Q = X^{1/2 - c/2}$. Then there is some $h \in H$ such that \[ \sum_{p \leq Q} \frac{\log p}{p} \frac{|S_p \cap (S_p + h)|}{p} < (\frac{1}{2} - c)  \log Q.\] 
\end{lemma}
\begin{proof}
Suppose not. Then certainly
\[ \sum_{h \in H}\sum_{p \leq Q} \frac{\log p}{p} \frac{|S_p \cap (S_p + h)|}{p} \geq (\frac{1}{2} - c) |H| \log Q \geq (\frac{1}{2} - c) |H| \sum_{p \leq Q} \frac{\log p}{p}.\]
Write $\mathscr{P}$ for the set of primes $p \leq Q$ such that
\begin{equation}\label{averagehyp}
\sum_{h \in H} \frac{|S_p \cap (S_p + h)|}{p} \geq (\frac{1}{2} - c^{1/2}) |H| .
\end{equation}
By Lemma \ref{simple-avg} applied with $X$ the set of primes $p \leq Q$, $\lambda(p) = \frac{\log p}{2p}|H|$, $f(p) = \frac{2}{|H|} \frac{1}{p}\sum_{h \in H} |S_p \cap (S_p + h)|$, $\epsilon = 2c$ and $\epsilon' = 2c^{1/2}$ we have\footnote{The alert reader will observe that our applications of Lemma \ref{simple-avg} are slightly bogus, since we have $f(p) \leq (p+1)/p$ rather than $f(p) \leq 1$, as required in the lemma. This can be corrected by instead setting $\lambda(p) = \frac{\log p}{2p} \frac{p+1}{p} |H|$ and $f(p) = \frac{2}{|H|} \frac{1}{p+1}\sum_{h \in H} |S_p \cap (S_p + h)|$, which makes no essential difference to the conclusions about $\mathscr{P}$ and $H_p$.}
\[ \sum_{p \in \mathscr{P}} \frac{\log p}{p} \geq (1 - c^{1/2} ) \sum_{p \leq Q} \frac{\log p}{p}.\]
Note that $|S_p \cap (S_p + h)| \leq \frac{1}{2}(p+1)$ always. It also follows from Lemma \ref{simple-avg} applied to the inequality (\ref{averagehyp}) that, for all $p \in \mathscr{P}$, there is a set $H'_p \subset H$ with $|H'_p| \geq (1 - c^{1/4}) |H|$ such that $|S_p \cap (S_p + h)| \geq (\frac{1}{2} - c^{1/4})p$ for all $h \in H'_p$. On the other hand, by Lemma \ref{add-comb} it follows that $|H'_p \md{p}| \leq 4 c^{1/4} p + 1$, and so all but $c^{1/4} |H|$ elements of $H$ reduce to lie in a set of size $4 c^{1/4} p + 1 < \frac{1}{3}p$ modulo $p$, for all $p \in \mathscr{P}$, a set which satisfies $\sum_{p \in \mathscr{P}} \frac{\log p}{p} \geq (1 - c^{1/2}) (\log Q + O(1))$. We may apply the larger sieve, Theorem \ref{larger-sieve}, to this situation, taking $\delta = 1 - c^{1/4}$ and $\sigma_p = 1/3$ for all $p \in \mathscr{P}$. This gives the bound
\[ |H| \ll \frac{Q}{3(1 - c^{1/4})^{2} (1 - c^{1/2})(\log Q + O(1)) -\log X  }\] provided that the denominator is positive. If $c$ is sufficiently small then the denominator will be positive with our choice of $Q$, namely $X^{1/2 - c/2}$, and we get the bound $|H| \ll X^{1/2 - c/2 + o(1)}$. This is contrary to assumption.
\end{proof}

We may now conclude the proof of Proposition \ref{diff-largersieve}. As in the hypothesis of the proposition, let $A \subset [X]$ be a set such that $A \md{p} \subset S_p$ for all $p$, where $|S_p| \leq \frac{1}{2}(p+1)$. Suppose additionally that $E(A,A) \geq |A|^3/K$. By Lemma \ref{lem3.5} there is a set $H \subset [-X,X]$, $|H| \geq |A|/2K$, such that $|A \cap (A + h)| \geq |A|/2K$ for all $h \in H$. If $|H| < X^{1/2 - c/4}$ then the proposition follows, so suppose this is not the case. Then Lemma \ref{lem2.8} applies and we may conclude that there is an $h \in H$ such that 
\[ \sum_{p \leq Q} \frac{\log p}{p} \frac{|S_p \cap (S_p + h)|}{p} < (\frac{1}{2} - c) \log Q,\] where $Q = X^{1/2 - c/2}$. 
Finally, by Lemma \ref{lem2.7}, it follows that $|A| \ll K X^{1/2 - c/4}$, thereby concluding the proof of the proposition.
\end{proof}

\section{Sieving by intervals}\label{interval-sieve-sec}

Our aim in this section is to establish Theorem \ref{interval-sieve}. We begin by recalling the statement of it. 
\begin{interval-sieve-repeat}
Suppose that $\A$ is a set of integers and that, for each prime $p$, the set $\A \mdlem{p}$ lies in some interval $I_p$. Let $\eps > 0$ be arbitrary. Then
\begin{enumerate} 
\item If $|I_p| \leq (1 - \eps) p$ for at least a proportion $\eps$ of the primes in each dyadic interval $[Z,2Z]$ then $|\A[X]| \ll_{\eps} (\log \log X)^{C\log(1/\eps)}$, where $C > 0$ is some absolute constant;
\item If $|I_p| \leq \frac{p}{2}$ for all primes then $|\A[X]| \ll (\log \log X)^{\gamma + o(1)}$, where $\gamma = \frac{\log 18}{\log(3/2)} \approx 7.129$;
\item If $|I_p| = [\alpha p, \beta p]$ for all primes $p$ and for fixed $0 \leq \alpha < \beta < 1$ \textup{(}not depending on $p$\textup{)} then $|\A| = O_{\beta - \alpha}(1)$;
\item There are examples of infinite sets $\A$ with $|I_p| \leq (\frac{1}{2} + \eps) p$ for all $p$.
\end{enumerate}
\end{interval-sieve-repeat}

The proof of parts (i) and (ii) relies on the following basic lemma.
\begin{lemma}\label{large-fourier} Suppose that $p$ is a prime, that $I_p \subset \Z/p\Z$ is an interval of length at most $(1 - \eps) p$, and that $A \subset [X]$ is a set with $A \mdlem{p} \subset I_p$. Then there is some integer $k$, $1 \leq k \leq \lceil 2/\eps^2 \rceil$, such that 
\[ |\sum_{x \leq X} 1_A(x) e(k x/p) | \geq \eps |A|/32.\] If $|I_p| \leq p/2$ then we have the following more precise conclusion: there is an integer $k \in \{1,2\}$ such that
\[ |\sum_{x \leq X} 1_A(x) e(k x/p)|  \geq |A|/3.\]
\end{lemma}
\begin{proof} 
We claim that there is a $1$-periodic real-valued function 
\[ f(\theta) = 1 + \sum_{0 < |k| \leq \lceil 2/\eps^2 \rceil} c_k e(k\theta)\]
such that $f(\theta) \leq 0$ when $|\theta| \geq \eps/2$. 

To construct $f(\theta)$, consider first the convolution $\psi(\theta) := 1_{|\theta| \leq \eps/4} \ast 1_{|\theta| \leq \eps/4} = \int_{\R/\Z} 1_{|\theta - \phi| \leq \eps/4} 1_{|\phi| \leq \eps/4} d\phi$. We have 
\[ |\hat{\psi}(k)| = |\widehat{1_{|\phi| \leq \eps/4}}(k)|^{2} = \big|\int_{\R/\Z} 1_{|\phi| \leq \eps/4} e(-k\phi) d\phi\big|^{2}  \leq \min(\eps, \frac{1}{\pi |k|})^2.\]
From the Fourier inversion formula it follows that 
\[ \frac{8}{\eps^2} \psi(\theta) = 2 + \sum_{k \neq 0} c_k e(k \theta),\] where $|c_k| \leq \min(8, \frac{1}{\eps^2 |k|^2})$. 
Furthermore, by construction, $\psi(\theta) = 0$ for $|\theta| \geq \eps/2$. Define 
\[ f(\theta) := 1 + \sum_{0 < |k| \leq K} c_k e(k \theta),\] where $K := \lceil 2/\eps^2 \rceil$. 
Since 
\[ \sum_{|k| > K} |c_k| \leq \sum_{|k| \geq K+1} \frac{1}{\eps^2 |k|^2} \leq \frac{2}{\eps^2 K} \leq 1,\] it follows that $f$ has the required properties. Now there is some $\beta \in [0,1]$ (depending on $I_p$) such that $\Vert \frac{x}{p} + \beta \Vert \geq \eps/2$ whenever $x \in A$, where $\Vert \cdot \Vert$ denotes distance to the nearest integer. This means that $f(\frac{x}{p} + \beta) \leq 0$, and so
\[ 1 + \sum_{0 < |k| \leq \lceil 2/\eps^2 \rceil} c_k e( k (\frac{x}{p} + \beta)) \leq 0.\]
It follows that
\[ |A| \leq -\sum_{0 < |k| \leq \lceil 2/\eps^2 \rceil} c_k \sum_{x \leq X} 1_A(x) e( k (\frac{x}{p} + \beta)).\]
Using the triangle inequality, one obtains
\[ |A| \leq \sum_{0 < |k| \leq \lceil 2/\eps^2 \rceil} |c_k| | \sum_{x \leq X} 1_A(x) e(k x/p) |.\]
To conclude the proof of the lemma, we observe that
\[ \sum_{0 < |k| \leq K} |c_k| \leq \frac{32}{\eps},\] an estimate that follows upon splitting into the ranges $0 < |k| \leq 1/\eps$ and $|k| > 1/\eps$.

For the second statement, simply note that the function $f(\theta) = 1 - 2\cos \theta + \cos 2\theta$ satisfies $f(\theta) \leq 0$ when $|\theta| \leq \pi/2$; rewriting the left-hand side as $2\cos\theta (\cos \theta - 1)$, this becomes clear. The rest of the argument proceeds as before.\end{proof}

We turn now to the proof of Theorem \ref{interval-sieve} (i). The general scheme of the argument, and in particular the use of Vinogradov's estimate (Proposition \ref{vino-est} below) was suggested to us by Jean Bourgain. We are very grateful to him for allowing us to include it here. The heart of the matter is the proof of the following lemma, from which Theorem \ref{interval-sieve} (i) follows rather easily by an iteration argument (or equivalently induction on $X$).

\begin{lemma}\label{it-step}
Suppose that $A \subset [X]$ and that $A \md{p}$ lies in an interval $I_p$ of length at most $(1-\eps) p$, for at least $\eps$ of all primes in each dyadic interval. Suppose that $X > X_0(\eps)$. Then there is a subinterval of $[X]$ of length $\exp (\log^{7/10} X)$ containing at least $c \eps^5 |A|$ points of $A$, where $c > 0$ is a small absolute constant.
\end{lemma}

Indeed, before proving this lemma let us explain how it implies Theorem \ref{interval-sieve} (i). We set $X_{0} = X$ and $A_{0} = \A[X]$, and by repeated application of the lemma we construct numbers $X_{i}$ and sets $A_{i}$ such that $A_{i} \subset [X_{i}]$, $A_{i} \md{p}$ lies in an interval $I_p$ of length at most $(1-\eps) p$, for at least $\eps$ of all primes in each dyadic interval, $\log X_{i+1} = \log^{7/10} X_{i}$ and $|A_{i+1}| \geq c \eps^5 |A_{i}|$. This procedure terminates when we first have $X_{i+1} \leq X_{0}(\eps)$, which will happen after $\ll \log\log\log X$ iterations. Consequently we have $|\A[X]| \leq (c^{-1}\eps^{-5})^{O(\log\log\log X)} X_{0}(\eps) \ll_{\eps} (\log \log X)^{C\log(1/\eps)}$, as claimed\footnote{As we have written things, we need to have $X_{0}(\eps) = \exp(C\log^{100}(1/\epsilon))$ (say) in order for the parameter $Y$ in the proof of Lemma \ref{it-step} to be large enough. But we remark that by taking more care of the final iterations in the proof of  Theorem \ref{interval-sieve} (i), one could obtain a bound $|\A[X]| \ll (\log \log X)^{C\log(1/\eps)}$ for all $X$, with an absolute implied constant (not depending on $\eps$).} in Theorem \ref{interval-sieve}.

\begin{proof}[Proof of Lemma \ref{it-step}] Suppose that $p$ is a prime such that $A \md{p} \subset I_p$. By Lemma \ref{large-fourier}, there is some $k$, $1 \leq k \leq \lceil 2/\eps^2 \rceil$, such that 
\begin{equation}\label{eq997} |\sum_{a \in A} e(k a/p)| \geq \eps |A|/32.\end{equation}
Let $Y$, $1 \ll Y \ll X$, be a parameter to be selected later (we will in fact take $Y = \exp(c\log^{7/10} X)$). We may choose a single $k$ so that the preceding estimate holds for $\gg \eps^2$ of the primes in $[Y, 2Y]$ for which we know that $A \md{p} \subset I_p$, that is for $\gg \eps^3$ of all the primes in $[Y, 2Y]$. Now we use the following fact: there is a weight function $w : [Y, 2Y] \rightarrow \R_{\geq 0}$ such that 
\begin{enumerate}
\item $w(p) \geq 1$ for all primes $p \in [Y, 2Y]$;
\item $\sum_{Y \leq n \leq 2Y} w(n) \leq 10 \pi(Y)$;
\item $w(n) = \sum_{d| n: d \leq Y^{1/2}} \lambda_d$, where $\sum_{d \leq Y} \frac{|\lambda_d|}{d} \ll \log^{3} Y$.
\end{enumerate}
Such a function can be constructed in the form $w(n) = (\sum_{d | n} \mu(d) \psi(\frac{\log d}{\log Y}))^2$, where $\psi \in C^{\infty}(\R)$ is supported on $|x| \leq \frac{1}{4}$, is bounded in absolute value by 1, and $\psi(0) = 1$. Property (i) is then clear, whilst bound (ii) can be verified by expanding out and interchanging the order of summation. To check (iii), we note that it is clear that  $|\lambda_d| \leq \sum_{[d_1, d_2] = d} 1 \leq \tau_3(d)$, the number of ways of writing $d$ as a product of three nonnegative integers.  The claimed bound is then an easy exercise. It follows from \eqref{eq997} and the above properties that 

\begin{equation}\label{eq527} \sum_{Y \leq n \leq 2Y} w(n) |\sum_{a \in A}  e(ka/n)|^2 \geq c \eps^5 \pi(Y) |A|^2.\end{equation} 

Expanding out and applying the triangle inequality yields

\begin{equation}\label{eq37} \sum_{a, a' \in A} |\sum_{Y \leq n \leq 2Y} w(n) e(\frac{k(a - a')}{n}) |  \geq c\eps^5 \pi(Y) |A|^2.\end{equation}
We now claim that, if $Y$ is chosen judiciously, the contribution to this from those pairs $a,a'$ with $|a - a'| \geq Y^{10}$ (say) can be ignored. Indeed suppose, on the contrary, that

\begin{equation}\label{eq843} |\sum_{Y \leq n \leq 2Y} w(n) e(\frac{x}{n})| \geq \frac{c}{100}\eps^5 \pi(Y),\end{equation}
for some $x := k(a - a')$ satisfying $Y^{10} \leq x \ll \eps^{-2}X$. By property (iii) of $w(n)$ and the triangle inequality, this implies that 

\[ \sum_{d \leq Y^{1/2}} |\lambda_d| |\sum_{Y/d \leq n' \leq 2Y/d} e(\frac{x}{dn'})| \geq \frac{c}{100} \eps^5 \pi(Y).\] By the upper bound (iii) for $\sum |\lambda_d| / d$ it follows that there is some $d \leq Y^{1/2}$ such that 

\begin{equation}\label{eq478} |\sum_{Y/d \leq n' \leq 2Y/d} e(\frac{x}{d n'})| \gg \eps^5 \log^{-4} Y \frac{Y}{d}.\end{equation}
At this point we invoke the following powerful estimate of Vinogradov.

\begin{proposition}\label{vino-est}
Let $\delta > 0$ be small and $Y$ be large, and suppose that $x \geq Y^5$. Suppose that $|\sum_{Y \leq n \leq 2Y} e(x/n)| \geq \delta Y$. Then $x \geq \exp( c\log^{3/2} Y/\log^{1/2}(1/\delta))$.
\end{proposition}
\begin{proof} Using e.g. Theorem 8.25 of Iwaniec and Kowalski \cite{iwaniec-kowalski}, one obtains that
\[ |\sum_{Y \leq n \leq 2Y} e(x/n)| \ll Y e^{-c\frac{\log^{3}Y}{\log^{2}x}}. \]
Thus we must have $1/\delta \gg \exp(c\log^{3}Y/\log^{2}x)$, from which the conclusion of the proposition quickly follows.
\end{proof}

Applying this Proposition to \eqref{eq478} leads to a contradiction unless 
\[ \frac{x}{d} > \exp \bigg(c \frac{\log^{3/2} Y}{(\log \log Y)^{1/2} + (\log(1/\epsilon))^{1/2}}\bigg), \]
which would imply that $X \gg \epsilon^{2} x > \exp (c (\log^{3/2} Y)/((\log \log Y)^{1/2} + (\log(1/\epsilon))^{1/2}))$. With $Y = \exp(c\log^{7/10} X)$ and $X \geq \exp(C\log^{100}(1/\epsilon))$, say, this will not be so. It follows that we were wrong to assume \eqref{eq843}, and so indeed the contribution to \eqref{eq37} of those pairs $a,a'$ with $|a - a'| \geq Y^{10}$ may be ignored. Thus we have

\begin{equation}\label{eq372} \sum_{\substack{a, a' \in A \\ |a - a'| \leq Y^{10}}} \bigg|\sum_{Y \leq n \leq 2Y} w(n) e(\frac{k(a - a')}{n}) \bigg|  \geq \frac{c}{2}\eps^5 \pi(Y) |A|^2.\end{equation}
Finally, we may apply the trivial bound to the inner sum, recalling from (ii) above that $\sum_{Y \leq n \leq 2Y} w(n) \leq 10 \pi(Y)$. We obtain
\[ \sum_{\substack{a, a' \in A \\ |a - a'| \leq Y^{10}}}  1 \gg \eps^5 |A|^2,\] 
which implies that there is a subinterval of $[X]$ of length $Y^{10}$ containing $\gg \eps^5 |A|$ elements of $A$. This concludes the proof of Lemma \ref{it-step}, and hence of Theorem \ref{interval-sieve} (i). \end{proof}

\emph{Proof of Theorem \ref{interval-sieve} (ii)} (sketch). We proceed as above, with the following changes.

\begin{itemize}
\item Use the second conclusion of Lemma \ref{large-fourier} to conclude that there is some $k \in \{1,2\}$ such that 
\[ \sum_{Y \leq p \leq 2Y} | \sum_{a \in A} e(ka/p)|^{2} \geq \frac{1}{18} (\pi(2Y) - \pi(Y)) |A|^{2}.\] This takes the place of \eqref{eq527}.
\item Expand out as in \eqref{eq37} to get
\[ \sum_{a, a' \in A} |\sum_{Y \leq p \leq 2Y} e(\frac{k(a - a')}{p})| \geq \frac{1}{18} (\pi(2Y) - \pi(Y)) |A|^2.\]
\item Choose $Y = \exp(\log^{2/3+o(1)} X)$, and use Jutila \cite[Theorem 2]{jutila} (which is a Vinogradov-type estimate for $\sum_{p \leq P} e(x/p)$) to show that the contribution from those pairs with $|a - a'| \geq Y^{10}$ can be ignored, so we have
\[ \sum_{\substack{a, a' \in A \\ |a - a'| \leq Y^{10}}} |\sum_{Y \leq p \leq 2Y} e(\frac{k(a - a')}{p})| \geq (\frac{1}{18} - o(1)) \pi(Y) |A|^2.\]
\item Conclude that there is some interval of length $\sim Y^{10}$ containing at least $(\frac{1}{18} - o(1))|A|$ points of $A$, and proceed iteratively as before.
\end{itemize}
\emph{Remark.} We could have used Jutila's bound in the proof of Theorem \ref{interval-sieve} (i) as well, instead of using the weight function $w$ . We chose not to do this in the interests of self-containment and of variety. Note that Jutila's paper predates Vaughan's identity \cite{vaughan} for prime number sums, and his argument would be a little more accessible if this device were used. A model for such an argument may be found in the paper of Granville and Ramar\'e \cite{granville-ramare}. \vspace{11pt}

\emph{Proof of Theorem \ref{interval-sieve}} (iii). This is essentially a consequence of Jutila \cite[Corollary, p126]{jutila}. A slight variant of that Corollary shows that the number of $p \in [x^{1/2}, 2x^{1/2}]$ with $\alpha \leq \{ x/p \} \leq \beta$ is $\sim (\beta - \alpha) (\pi(2x^{1/2}) - \pi(x^{1/2}))$, and so all elements of $\A$ are bounded by $O_{\beta - \alpha}(1)$.\vspace{11pt}

\emph{Proof of Theorem \ref{interval-sieve}} (iv). We take $\A$ to consist of the numbers $a_i = \prod_{p \leq X_i} p$, for some extremely rapidly-growing sequence $X_1 < X_2 < \dots$. Given a prime $p$, suppose that $X_i \leq p \leq X_{i+1}$. Then $a_{i+1},a_{i+2},\dots$ all reduce to zero $\md{p}$, and so $\A \md{p} = \{0,a_1,\dots, a_i\}$. By choosing the $X_i$ sufficiently rapidly growing we may ensure that $0 < a_1 < \dots < a_{i-1} < \eps p$. Regardless of the value of $a_i \md{p}$ (which we cannot usefully control) the set $\A \md{p}$ will be contained in some interval of length at most $(\frac{1}{2} + \eps) p$.\vspace{11pt}

\emph{Remark.} With $\A$ as constructed above, the shape of $|\A[X]|$ is $\log_* X$. Thus there is still a considerable gap between the bound of (i) and the construction given here.  We expect, however, that the correct bound in (i) is of $\log_* $ type, which would follow assuming vaguely sensible conjectures on exponential sums $\sum_{n \leq Y} e(x/n)$. If, for example, the conclusion of Proposition \ref{vino-est} were instead that $x \geq \exp(Y^{1/10})$ then we would get a $\log_*$-type bound on $\A[X]$ in this case.

\section{Sieving by arithmetic progressions}

In this section we shall prove Theorem \ref{prog-thm}, whose statement was as follows.

\begin{prog-thm-rpt}
Let $\eps > 0$. Suppose that $\A$ is a set of integers and that, for each prime $p \leq X^{1/2}$, the set $\A \mdlem{p}$ lies in some arithmetic progression $S_p$ of length $(1-\eps)p$. Then $|\A[X]| \ll_{\eps} X^{1/2 - \eps'}$, where $\eps' > 0$ depends on $\eps$ only.
\end{prog-thm-rpt}

Suppose to begin with that $|S_{p}|=\frac{1}{2}(p+1)$ for each odd prime $p$. Then (dilating $S_p$ to the interval $\{1,\dots, \frac{1}{2}(p+1)\}$, which preserves the additive energy)
$$ E_{p}(S_{p},S_{p}) = (\frac{p+1}{2})^{2} + 2\sum_{h=1}^{\frac{1}{2}(p-1)} (\frac{p+1}{2} - h)^{2} \geq \frac{p^{3}}{12} . $$
Thus Theorem \ref{thm1.4} is applicable with the choice $\delta = 1/48$, and the result follows in this case.

Now we turn to the proof of Theorem \ref{prog-thm} for arbitrary $\eps > 0$. We begin with a result which should be compared to Corollary \ref{cor1.3}, but which is simpler to state and prove than that result.

\begin{lemma}\label{progr-energy}
Let $\eps > 0$ be small, and suppose that $A \subset [X]$ is a set and $\mathscr{P} \subset [X^{1/2}]$ is a set of primes such that $|\mathscr{P}| \geq 2/\eps^{2}$. If $S_{p} \subset \Z/p\Z$ is an arithmetic progression of length at most $(1-\eps)p$, and if $A \mdlem{p} \subset S_p$ for each prime $p \in \mathscr{P}$, then $E(A,A) \gg \frac{\eps^{4} |A|^{4}}{X} |\mathscr{P}|$,
where the constant implicit in the $\gg$ notation is absolute.  
\end{lemma}

\begin{proof}
In view of Lemma \ref{lift-lem}, and our assumption that $|\mathscr{P}| \geq 2/\eps^{2}$, it will suffice to show that
$$ E_{p}(A,A) - \frac{|A|^{4}}{p} \gg \frac{\eps^{4} |A|^{4}}{p} $$
for each prime $p \in \mathscr{P}$ that is greater than $2/\eps^{2}$ (that being a positive proportion of all the primes in $\mathscr{P}$). As in the proof of Corollary \ref{cor1.3} we have
$$ E_{p}(A,A)  = p^3 \sum_{r \in \Z/p\Z} \big|\frac{1}{p} \sum_{x \leq X} 1_A(x) e(-xr/p) \big|^{4} . $$
The contribution from the $r=0$ term is evidently equal to $|A|^{4}/p$. By Lemma \ref{large-fourier} (which was stated in the case $S_p$ is an interval, but may easily be adapted to the case $S_p$ a progression by dilation), if $p > \lceil 2/\eps^2 \rceil$ then there is some nonzero $r$ satisfying $|\sum_{x \leq X} 1_A(x) e(rx/p)| \geq \eps |A|/32$. The result follows immediately.
\end{proof}

The other major ingredient that we shall need is an analogue of Lemma \ref{lem2.8} that applies when the sets $S_{p}$ have size at most $(1-\eps)p$, rather than size at most $\frac{1}{2}(p+1)$ as in that lemma. The following result provides this.
\begin{lemma}\label{progr-goodh}
Let $c > 0$ be a sufficiently small absolute constant. Let $H \subset [-X,X]$ be a set of size $X^{1/2 - c/4}$, and let $Q = X^{1/2 - c/2}$. Suppose that for each prime $p \in \mathscr{P}'$ we have a subset $S_p \subset \Z/p\Z$ such that $\frac{1}{10}p \leq |S_p| \leq (1-\eps)p$, where $\mathscr{P}'$ is a subset of the primes $p \leq Q$ satisfying $\sum_{p \in \mathscr{P}'} \frac{\log p}{p} \geq \frac{1}{3}\log Q$.
Then there is some $h \in H$ such that \[ \sum_{p \in \mathscr{P}'} \frac{\log p}{p} \frac{|S_p \cap (S_p + h)|}{p} < (1 - c\eps) \sum_{p \in \mathscr{P}'} \frac{\log p}{p} \frac{|S_p|}{p} .\] 
\end{lemma}

\begin{proof}
The proof is quite close to the proof of Lemma \ref{lem2.8}, so we shall give a fairly brief account. If the conclusion were false then we would have
$$ \sum_{h \in H} \sum_{p \in \mathscr{P}'} \frac{\log p}{p} \frac{|S_p \cap (S_p + h)|}{p} \geq (1 - c\eps) |H| \sum_{p \in \mathscr{P}'} \frac{\log p}{p} \frac{|S_p|}{p}. $$
In this case, if we let $\mathscr{P}$ denote the subset of primes $p \in \mathscr{P}'$ for which
$$ \sum_{h \in H} \frac{|S_p \cap (S_p + h)|}{p} \geq (1 - c^{1/2}\eps) |H| \frac{|S_p|}{p} , $$
then applying Lemma \ref{simple-avg} with $X = \mathscr{P}'$, $\lambda_{p} = \frac{\log p}{p^2} |H| |S_p|$ and $f(p) = |H|^{-1}\sum_{h \in H} \frac{|S_p \cap (S_p + h)|}{|S_p|}$ yields
$$ \sum_{p \in \mathscr{P}} \frac{\log p}{p} \frac{|S_p|}{p} \geq (1 - c^{1/2}) \sum_{p \in \mathscr{P}'} \frac{\log p}{p} \frac{|S_p|}{p} \geq \frac{1}{40}\log Q. $$
As in the proof of Lemma \ref{lem2.8}, it also follows from Lemma \ref{simple-avg} that, for all $p \in \mathscr{P}$, there is a set $H'_p \subset H$ with $|H'_p| \geq (1 - c^{1/4}) |H|$ such that $|S_p \cap (S_p + h)| \geq (1 - c^{1/4}\eps)|S_p|$ for all $h \in H'_p$. Finally, using Lemma \ref{add-comb} applied to $S_{p}$ (and with $\epsilon$ replaced by $c^{1/4}\epsilon$) it follows that $|H'_p \md{p}| \leq 4 c^{1/4} \eps p + 1$, and so all but $c^{1/4} |H|$ elements of $H$ reduce to lie in a set of size $4 c^{1/4} \eps p + 1 < \frac{1}{100}p$ modulo $p$, for all $p \in \mathscr{P}$, a set which satisfies $\sum_{p \in \mathscr{P}} \log p/p \geq \frac{1}{40}\log Q$. We may apply the larger sieve, Theorem \ref{larger-sieve}, to this situation, taking $\delta = 1 - c^{1/4}$ and $\sigma_p = 1/100$ for all $p \in \mathscr{P}$. This gives the bound
\[ |H| \ll \frac{Q}{100(1 - c^{1/4})^{2} (1/40)\log Q -\log X  }\] provided that the denominator is positive. If $c$ is sufficiently small then this will be so with our choice of $Q$, namely $X^{1/2 - c/2}$, and we get the bound $|H| \ll X^{1/2 - c/2 + o(1)}$. This is contrary to assumption.
\end{proof}

Now we can prove Theorem \ref{prog-thm}, by applying the foregoing lemmas repeatedly to the intersection of $A$ (and of the sets $S_{p}$) with shifted copies of itself. The key point here is that, since any subset of $A$ will lie in the arithmetic progression $S_{p}$ when reduced modulo $p$, we can use Lemma \ref{progr-energy} throughout this ``intersecting process'' to obtain a lower bound on additive energy. In particular, we don't need to keep track of any ``uniformity of fibres'' throughout the process. Eventually we will obtain a subset of $A$ that has cardinality quite close to $|A|$, but lies modulo $p$ in a multiply intersected copy of $S_p$ having size $< (1/2-c)p$ (for most primes $p$). The theorem will then follow immediately from the larger sieve.

\begin{proof}[Proof of Theorem \ref{prog-thm}]
We assume that $\eps$ is small, and that $X$ is sufficiently large in terms of $\eps$. This is certainly permissible for proving the theorem. We will prove that $|A| < X^{1/2-c(\eps)}$, where $c(\eps) = K^{-1/\eps}$ for a large absolute constant $K > 0$. Suppose, for a contradiction, that $|A| \geq X^{1/2-c(\eps)}$. We proceed iteratively, setting $A_0 = A$ and constructing a sequence of sets $A_i \subset A$, $i = 1,2,3,\dots$ such that
\begin{equation}\label{progr-size}
|A_i| \geq X^{1/2-3^{i}K^{-1/\eps}}.
\end{equation}
The sets $A_i$ will satisfy $A_i \md{p} \subset S_p^i$, where $S_p^i \subset S_p$, and where
\begin{equation}\label{progr-decay}
\sum_{p \leq Q} \frac{\log p}{p} \frac{|S_p^i|}{p} < (1-c\eps/2)^{i}(\log Q + O(1)).
\end{equation}
Here $c$ is the absolute constant from Lemma \ref{progr-goodh}, and $Q=X^{1/2-c/2}$. After $O(1/\eps)$ steps we will, in particular, have
$$ \sum_{p \leq Q} \frac{\log p}{p} \frac{|S_p^i|}{p} < (1/2 - c)\log Q , $$
and therefore, writing $\sigma^i_p := |S_p^i|/p$, we will have
$$ \sum_{p \leq Q} \frac{\log p}{\sigma_p^i p} \geq 4\log Q + O(1) - 4(\frac{1}{2} - c)\log Q = (2 + 4c)\log Q + O(1) , $$
as argued in the proof of Lemma \ref{lem2.7}. Using the larger sieve, this implies that $|A_i| \ll X^{1/2-c/2}$, which contradicts the lower bound (\ref{progr-size}) provided that $K$ was chosen large enough.

We will show how the set $A_{i+1}$ is obtained from $A_i$, and verify that it satisfies the size bound (\ref{progr-size}) and that its reductions modulo primes $p$ satisfy the bound (\ref{progr-decay}). Firstly, if $A_i \subset A$ satisfies (\ref{progr-size}) then Lemma \ref{progr-energy} implies that
$$ E(A_i,A_i) \gg \frac{\eps^{4} |A_i|}{X^{1/2} \log X} |A_i|^3 \gg \frac{\eps^{4} X^{-3^{i}K^{-1/\eps}}}{\log X} |A_i|^3 \geq 2X^{-2 \cdot 3^{i}K^{-1/\eps}} |A_i|^3 , $$
provided that $X$ is large enough in terms of $\eps$. Using Lemma \ref{lem3.5}, it follows that there is a set $H \subset [-X,X]$ such that $|H| \geq |A_i|X^{-2 \cdot 3^{i}K^{-1/\eps}} \geq X^{1/2-3^{i+1}K^{-1/\eps}}$, and such that
$$ |A_i \cap (A_i + h)| \geq X^{1/2-3^{i+1}K^{-1/\eps}} $$
for all $h \in H$. The set $A_{i+1}$ will be of the form $A_i \cap (A_i + h)$, for suitably chosen $h \in H$. Note that any such choice will indeed satisfy the size bound (\ref{progr-size}). In view of (\ref{progr-decay}), we may assume that the sets $S_p^i$ satisfy
$$ (1/2 - c)\log Q \leq (1-c\eps/2)^{i+1}(\log Q + O(1)) \leq \sum_{p \leq Q} \frac{\log p}{p} \frac{|S_p^i|}{p} < (1-c\eps/2)^{i}(\log Q + O(1)), $$
the lower bound holding because if it failed we could simply set $A_{i+1}=A_i$. Now let $\mathscr{P}'$ denote the set of primes $p \leq Q$ for which $|S_p^i| \geq \frac{1}{10}p$. We must have
\[ \sum_{p \in \mathscr{P}'} \frac{\log p}{p} \geq \frac{1}{3}\log Q \qquad \textrm{and} \qquad \sum_{p \in \mathscr{P}'} \frac{\log p}{p} \frac{|S_p^i|}{p} \geq \frac{1}{2} \sum_{p \leq Q} \frac{\log p}{p} \frac{|S_p^i|}{p}, \]
say, because otherwise the lower bound we just assumed would be violated. Thus we can apply Lemma \ref{progr-goodh}, deducing that for some $h \in H$ we have
$$ \sum_{p \in \mathscr{P}'} \frac{\log p}{p} \frac{|S_p^i \cap (S_p^i + h)|}{p} < (1 - c\eps) \sum_{p \in \mathscr{P}'} \frac{\log p}{p} \frac{|S_p^i|}{p} \leq \sum_{p \in \mathscr{P}'} \frac{\log p}{p} \frac{|S_p^i|}{p} - \frac{c\eps}{2}\sum_{p \leq Q} \frac{\log p}{p} \frac{|S_p^i|}{p} . $$
Finally, if we set $A_{i+1} = A_i \cap (A_i + h)$ for this choice of $h$, and correspondingly set $S_p^{i+1} = S_p^i \cap (S_p^i + h)$, then the upper bound (\ref{progr-decay}) will indeed be satisfied.
\end{proof}

\section{Robustness of the inverse large sieve problem}\label{stab-sec}

In this section we prove Theorem \ref{stability-thm}. Let us begin by recalling the statement. 

\begin{stability-thm-repeat}
Let $X_0 \in \N$, and let $\eps > 0$. Let $X \in \N$ be sufficiently large in terms of $X_0$ and $\eps$, and suppose that $H \leq X^{1/8}$. Suppose that $A,B \subset [X]$ and that $|A \md{p}| + |B \md{p}| \leq p+1$ for all $p \in [X_0, X^{1/4}]$. Then, for some absolute constant $c > 0$, one of the following holds:
\begin{enumerate}
\item \textup{(Better than large sieve)} Either $|A \cap [X^{1/2}]|$ or $|B \cap [X^{1/2}]|$ is $\leq X^{1/4 - c\eps^3}$;
\item \textup{(Behaviour with quadratics)} Given any two rational quadratics $\psi_A,\psi_B$ of height at most $H$, either $|A \setminus \psi_{A}(\Q)|$ and $|B\setminus \psi_{B}(\Q)| \leq HX^{1/2 - c}$, or else at least one of $|A \cap \psi_{A}(\Q)|$ and $|B \cap \psi_{B}(\Q)|$ is bounded above by $HX^{1/4 + \eps}$.
\end{enumerate}
\end{stability-thm-repeat}

The proof of Theorem \ref{stability-thm} requires a number of preliminary ingredients, which we assemble now. We start with some results concerning \emph{quasisquares}, that is to say squarefree integers that are quadratic residues modulo ``many'' primes (see, for example, \cite[Section 12.14]{dekoninck-luca}). The first result is not the one actually required later on (which is Lemma \ref{pseudosquares}), but it may be of independent interest and its proof motivates the proof of the result needed later.

\begin{lemma}\label{quasi-square}
Let $\eta > 0$, and suppose that $Y \geq Z > 2$. Suppose that $\mathscr{P} \subset [Z,2Z]$ is a set consisting of at least $\eta Z/\log Z$ of the primes in $[Z,2Z]$. Then the number of squarefree $q \in [1,Y]$ such that $\left( \frac{q}{p} \right) = 1$ for all $p \in \mathscr{P}$ is at most $8 (6 \log Y/\eta)^{3\log Y/\log Z}$.
\end{lemma}
\begin{proof}
First of all, let $\mathscr{Q}$ denote the set of all $q$ that are squarefree, lie in $[1,Y]$ and are a quadratic residue modulo all $p \in \mathscr{P}$. Each $q \in \mathscr{Q}$ is either congruent to $1 \md{4}$, or it is congruent to $2$ or $3 \md{4}$. Let us assume that at least half of the elements of $\mathscr{Q}$ are congruent to $2$ or $3 \md{4}$, and henceforth redefine $\mathscr{Q}$ to consist of those elements only, and aim to show that $|\mathscr{Q}| \leq 4 (6 \log Y/\eta)^{3\log Y/\log Z}$. (The proof when at least half of the elements are congruent to $1 \md{4}$ is essentially the same.)

Let $k \geq 3$ be the smallest integer for which $Z^k  > Y^2$. Then if $n \in [Z^k, 2^k Z^k]$ is any product of $k$ distinct primes from $\mathscr{P}$, and if $q \in \mathscr{Q}$, we have that the Jacobi symbol $\left( \frac{q}{n}  \right) = \prod_{p \mid n} \left( \frac{q}{p}  \right) = 1$. 
Let $\mathscr{S}$ be the set of all such $n$; then clearly
\begin{equation}\label{s-lower} |\mathscr{S}| = \binom{|\mathscr{P}|}{k} \geq \frac{|\mathscr{P}|^{k}}{k^{k}} \geq \left(\frac{\eta}{k\log Z}\right)^k Z^k . \end{equation}
(Of course this is only true if $|\mathscr{P}| \geq k$, but otherwise the bound we shall derive is trivial anyway.)

Finally, note that if $q$ is squarefree and congruent to $2$ or $3 \md{4}$, and if $n \in \mathscr{S}$ (so, in particular, $n$ is odd), then $\left( \frac{q}{n} \right) = \left( \frac{4q}{n} \right) = \chi_{4q}(n)$, where $\chi_{4q}(n)$ is a primitive character modulo $4q$. (It is the primitive quadratic character corresponding to the fundamental discriminant $4q$.) The multiplicative form of the large sieve \cite[Theorem 7.13]{iwaniec-kowalski} implies that
\[ \sum_{\substack{q \leq Q, \\ q \equiv 2 \text{ or } 3 \md 4, \\ q \; \text{squarefree}}} \left|\sum_{n \in \mathscr{S}} a_{n} \chi_{4q}(n) \right|^2 \leq (16Q^2 + N) \sum_{n \in \mathscr{S}} |a_{n}|^{2} , \]
for any set $\mathscr{S} \subset [N]$ and for any $Q$ and any coefficients $a_{n}$. Applying this with our particular set $\mathscr{S}$, and with $a_{n} = 1$, yields
\[ |\mathscr{Q}||\mathscr{S}|^2 \leq (16Y^2 + 2^k Z^k)|\mathscr{S}| < 2^{k+2} Z^k |\mathscr{S}|,\] and therefore by \eqref{s-lower}
\[ |\mathscr{Q}| \leq \frac{2^{k+2} Z^k}{|\mathscr{S}|} \leq 4 (\frac{2k\log Z}{\eta})^k.\] Noting that $k \leq \frac{2\log Y}{\log Z} + 1 \leq \frac{3\log Y}{\log Z}$, the result follows.
\end{proof}
\emph{Remarks.} The conclusion of Lemma \ref{quasi-square} is nontrivial when $Y$ is any fixed power of $Z$, and even for somewhat larger values of $Y$. It seems to us that the bound obtained here is stronger than could (straightforwardly) be obtained using the real character sum estimate of Heath-Brown \cite{hb}, which comes with an unspecified factor of $(QN)^{\eps}$.

Now we present the result we need later on. Of course more general statements are possible, but we leave their formulation as an exercise to the interested reader.

\begin{lemma}\label{pseudosquares}
Suppose that $Z \geq Y^{1/5}$ are large. The number of squarefree $q \in [1,Y]$ which are squares modulo 95\% of the primes in $[Z, 2Z]$ is $O(\log^{10} Z)$. Furthermore, all of these $q$ apart from $q=1$ are at least $Y^\varrho$, for a certain absolute constant $\varrho > 0$.
\end{lemma}
\begin{proof}
The argument for the first part closely follows the preceding. Write $\mathscr{Q}$ for the set of all squarefree $q \in [1,Y]$ which are squares modulo at least 95\% of the primes $p \in [Z,2Z]$. Write $\mathscr{P}_q$ for the set of these primes; note carefully that $\mathscr{P}_q$ may depend on $q$.  Write $\mathscr{S}$ for the set of products of $10$ distinct primes from $[Z,2Z]$, and  write $\mathscr{S}_q$ for the set of products of $10$ distinct primes from $\mathscr{P}_q$. Note that $\left(\frac{q}{n} \right) = 1$ whenever $n \in \mathscr{S}_q$. Furthermore,
\[ |\mathscr{S}_q| = \binom{|\mathscr{P}_q|}{10} \geq (1 - o(1)) (1 - \frac{1}{20})^{10} \binom{|\mathscr{P}|}{10} > 0.59\binom{|\mathscr{P}|}{10} = 0.59|\mathscr{S}|\] for every $q \in \mathscr{Q}$ (the key point here is that $0.59 > 0.5$). It follows that for every $q \in \mathscr{Q}$ we have $\sum_{n \in \mathscr{S}} \left( \frac{q}{n} \right) \geq 0.18|\mathscr{S}|$. The proof now concludes as before.

To prove the second part, we use a form of the prime number theorem for the real character $\chi(m) = \left( \frac{q}{m} \right)$ (which, provided $q > 1$ is squarefree, is always a non-principal character of conductor at most $4q$). This tells us (see e.g. Theorem 7 in Gallagher's paper~\cite{gallagherdensity}) that 
\[ \sum_{Z \leq m \leq 2Z} \Lambda(m) \chi(m) = O(Z \max ( e^{-c\sqrt{\log Z}}, e^{-c\log Z/\log q} )) \] if $\chi$ has no exceptional zero, and
\[ \sum_{Z \leq m \leq 2Z} \Lambda(m) \chi(m) = \frac{Z^{\beta} - (2Z)^{\beta}}{\beta} + O(Z \max ( e^{-c\sqrt{\log Z}}, e^{-c\log Z/\log q} )),\] if $\chi$ does have an exceptional zero $\beta \in (\frac{1}{2}, 1)$. Either way, we have the 1-sided inequality
\[ \sum_{Z \leq m \leq 2Z} \Lambda(m) \chi(m) \leq O(Z \max ( e^{-c\sqrt{\log Z}}, e^{-c\log Z/\log q} ) ).\]
Since
\[ \sum_{Z \leq m \leq 2Z} \Lambda(m) = Z(1 + o(1))\] by the prime number theorem, it follows that if $q \leq Y^{\varrho}$ with $\varrho$ small enough then at most $95$\% of the primes in $[Z,2Z]$ are such that $(q | p) = \chi(p) = 1$ .\end{proof}

\begin{proof}[Proof of Theorem \ref{stability-thm}]  A key role will be played by primes $p$ that are close to $X^{1/4}$. It is convenient to introduce some terminology concerning them. Let $C$ be a large absolute constant to be specified later. We say that $p \sim X^{1/4}$ if $X^{1/4 - C\eps} < p < X^{1/4}$. Furthermore, we will say that a certain property holds ``for at least 1\% of primes $p \sim X^{1/4}$'' if the set $\mathscr{P}$ of primes for which this property holds satisfies
\[ \sum_{p \sim X^{1/4} : p \in \mathscr{P}} \frac{\log p}{p} \geq 0.01 \sum_{p \sim X^{1/4}} \frac{\log p}{p}.\] The weighting of $\log p / p$ is included with some later applications of the larger sieve in mind. We begin with some preliminary analysis using the larger sieve, strongly based on the work of Elsholtz \cite{elsholtz}. 

\begin{lemma}\label{large-sieve-app} Let $A, B$ be sets such that $|A \md{p}| + |B \md{p}| \leq p+1$ for all primes $p \in [X_0, X^{1/4}]$. Let $\eps > 0$, and suppose that $X$ is sufficiently large in terms of $X_0$ and $\eps$. Then either $A \cap [X^{1/2}]$ or $B \cap [X^{1/2}]$ has size at most $X^{1/4 - c\eps^3}$, or else for at least 99\% of primes $p \sim X^{1/4}$ we have both $|A \md{p}| \leq (\frac{1}{2} + \eps)p$ and $|B\md{p}| \leq (\frac{1}{2} + \eps) p$. We may take $c = 2^{-10}$.
\end{lemma}
\begin{proof} 
Write $\alpha_p := |A \md{p}|/p$, $\beta_p := |B \md{p}|/p$. Thus we are assuming that $\alpha_p + \beta_p \leq 1 + \frac{1}{p}$.
 Suppose the final statement of the lemma is false. Then 
\[ \sum_{p \sim X^{1/4}} \frac{\log p}{p} \big( (1 - 2\alpha_p)^2 + (1 - 2\beta_p)^2\big) \geq 2^{-5}\eps^2 \sum_{p \sim X^{1/4}} \frac{\log p}{p} \geq 2^{-5} \eps^3 \log X.\]
If $c < 2^{-8}$, we can remove the contribution from $X^{1/4 - 2c\eps^3} \leq p \leq X^{1/4}$ trivially to get
\[ \sum_{p < X^{1/4 - 2c\eps^3}} \frac{\log p}{p}\big( (1 - 2\alpha_p)^2 + (1 - 2\beta_p)^2\big) \geq 2^{-6} \eps^3 \log X.\]
We claim that if $a,b$ are positive real numbers with $a + b \leq 1$ then 
\[ \frac{1}{a} + \frac{1}{b} \geq 4 + 2(1 - 2a)^2 + 2(1 - 2b)^2.\]
To see this, apply the inequality
\[ \frac{1}{x} + \frac{1}{1 - x} = 4 + \frac{(1 - 2x)^2}{x (1 - x)} \geq 4 + 4(1 - 2x)^2\] with $x = a$ and $x = b$ in turn, and add the results.

Applying this together with the above, we obtain
\[ \sum_{X_0 \leq p \leq X^{1/4 - 2c\eps^3}} \frac{\log p}{p} (\frac{1}{\alpha_p} + \frac{1}{\beta_p}) \geq (1 - 8c\eps^3)\log X + 2^{-5}\eps^3 \log X - O(1) \geq (1-  8c\eps^3 + 2^{-6} \eps^3)\log X .\]
Here we used the fact (an estimate of Mertens) that $\sum_{X_0 \leq p \leq Z} \frac{\log p}{p} = \log Z + O_{X_0}(1)$. Note also that we only have $\alpha_p + \beta_p \leq 1 +\frac{1}{p}$, and not $\alpha_p + \beta_p \leq 1$; the introduction of the $O(1)$ term takes care of this as well, the full justification of which we leave to the reader\footnote{We are working with the condition $|A \md{p}| + |B \md{p}| \leq p+1$, rather than the cleaner condition $|A \md{p}| + |B \md{p}| \leq p$, so that we can formulate Theorem \ref{stability-thm-symmetric} to include the case in which $\A$ is the set of values of a quadratic. Note, however, that in this case Lemma \ref{large-sieve-app} is vacuous anyway. Therefore this small point really can be ignored.}.
Without loss of generality the contribution from the $\alpha_p$ is at least that from the $\beta_p$, so 
\[ \sum_{X_0 \leq p \leq X^{1/4 - 2c\eps^3}} \frac{\log p}{p} \frac{1}{\alpha_p} \geq  (\frac{1}{2}  - 4c\eps^3+ 2^{-7}\eps^3)\log X > (\frac{1}{2} + 2^{-8}\eps^3) \log X\] if $c \leq 2^{-10}$.
Then, however, the larger sieve implies that
\[ |A \cap [X^{1/2}]| \ll \frac{X^{1/4 - 2c\eps^3}}{\eps^3 \log X}  < X^{1/4 - c\eps^3},\] and the result follows.
\end{proof}

Suppose now that the hypotheses are as in Theorem \ref{stability-thm}. Replace $\eps$ by $\eps/2C$ (the statement of the theorem does not change). Let $\psi_A,\psi_B$ be rational quadratics of height at most $H$. If option (ii) of the theorem does not hold then at least $HX^{1/4 + 2C\eps}$ elements of $A$ lie in $\psi_{A}(\Q)$ and at least $HX^{1/4 + 2C\eps}$ elements of $B$ lie in $\psi_{B}(\Q)$. Suppose also that option (i) of Theorem \ref{stability-thm} does not hold. Then we may apply Lemma \ref{large-sieve-app} to conclude that both $|A \md{p}|$ and $|B \md{p}|$ are $\leq (\frac{1}{2} + \eps)p$ for at least 99\% of all primes $p \sim X^{1/4}$. (We urge the reader to recall the special meaning of this notation.) Using this information together with the fact that $|A \md{p}| + |B \md{p}| \leq p+1$ for all $p \sim X^{1/4}$, we will deduce that in fact
\begin{equation}\label{weak} |A \setminus \psi_{A}(\Q)|, |B \setminus \psi_{B}(\Q)| \ll HX^{1/2 - c},  \end{equation} this being the other conclusion of Theorem \ref{stability-thm} (ii).  It suffices to prove this for $A$, the proof for $B$ being identical. Write $\psi = \psi_{A}$, and suppose that $p \sim X^{1/4}$. Set
\[ T_p := A \md{p} \cap (\psi(\Q) \cap \Z)  \md p,\]
\[ U_p := A \md{p} \setminus (\psi(\Q) \cap \Z) \md{p}.\]
We know that $|A \md{p}| \leq (\frac{1}{2} + \eps) p$ for at least 99\% of all primes $p \sim X^{1/4}$. For these primes, then,
\[ |T_p| + |U_p| \leq (\frac{1}{2} + \eps)p.\]
We claim that $|U_p| \leq 2\eps p$ for at least 98\% of all primes $p \sim X^{1/4}$. If this failed, we would have \begin{equation}\label{tp-bound}|T_p| \leq (\frac{1}{2} - \eps) p\end{equation} for at least $1\%$ of the primes $p \sim X^{1/4}$. Write $\psi(x) = \frac{1}{d}(a x^2 + bx + c)$ with $a,b,c,d$ having no common factor and $|a|, |b|, |c|, |d| \leq H$, and note that $\psi^{-1}(\Z) \subset \frac{1}{a}\Z$. 
Note furthermore that
\begin{equation}\label{35} \{ x \in \Z : \psi(\frac{x}{a}) \in A \cap \psi(\Q)\} \subset \bigcap_{p \sim X^{1/4}} \{ x \in [-C_1HX^{1/2}, C_1HX^{1/2}] : \psi(\frac{x}{a}) \md{p} \in T_p\} ,\end{equation}  where $C_1$ is some absolute constant. If $p \sim X^{1/4}$, the condition that $\psi(\frac{x}{a}) \md{p} \in T_p$ forces $x \md{p}$ to lie in some set $S_p \subset \Z/p\Z$ of residue classes with $|S_p| \leq 2|T_p|$. We now have a large sieve problem to which Proposition \ref{classicls} may be applied. If $p$ is such that \eqref{tp-bound} holds then we have $|S_p^c|/|S_p| \geq \eps$, and so
\[ \sum_{q \leq X^{1/4}} \mu^2(q) \prod_{p | q} \frac{|S_p^c|}{|S_p|} \geq \sum_{p \sim X^{1/4}} \frac{|S_p^c|}{|S_p|} \gg \eps \frac{X^{1/4 -C\eps}}{\log X}.\] Here, $X^{1/4 - C\eps}/\log X$ is a crude lower bound for the size of a set of primes constituting 1\% of all $p \sim X^{1/4}$, this lower bound being attained when all the primes congregate at the bottom of the interval $X^{1/4 - C\eps} \leq p \leq X^{1/4}$.
It follows from \eqref{35} and Proposition \ref{classicls} that $|A\cap \psi(\Q)|  < HX^{1/4 + 2C\eps}$ if $X$ is large enough, contrary to assumption. 

Now let us write $A= A_{\psi} \cup E$, where $A_{\psi}$ consists of those $x \in A$ for which $x \md{p} \in T_p$ for at least $97\%$ of $p \sim X^{1/4}$, and
$E$ consists of those $x \in A$ such that $x \md{p} \in U_p$ for at least $3\%$ of $p \sim X^{1/4}$. The idea here is that $A_{\psi}$ satisfies a large number of local conditions suggesting that its elements lie in $\psi(\Q)$. We would like to relate $A_{\psi}$ to $A \cap \psi(\Q)$, and show that $E$ is small. With this idea in hand, we can divide the task of proving \eqref{weak} into two subclaims, namely
\begin{equation}\label{claim} |A_{\psi} \setminus \psi(\Q)| \ll H X^{1/2 - c} \qquad \mbox{and} \qquad |E| \ll X^{1/4} .\end{equation}
Of course, we could tolerate a weaker bound for $|E|$, but as it turns out we need not settle for one.

We start with the first claim, which is quite straightforward given results we established earlier. Let $\Delta$ be the discriminant of $\psi$. If $x$ is an integer then so is $4ad x + \Delta d^2$, and furthermore if $x =\psi(n)$ then $4a d x + \Delta d^2 = (2a n + b)^2$. Therefore if $x \in A_{\psi}$ then $4a d x + \Delta d^2$ is an integer which is a square modulo $p$ for at least 97\% of all $p \sim X^{1/4}$. By a simple averaging argument, $4ad x + \Delta d^2$ is a square modulo at least 95\% of all $p\in [Z,2Z]$ for some $Z \sim X^{1/4}$. It follows from Lemma \ref{pseudosquares} that \begin{equation}\label{false-eq}4a d x + \Delta d^2 = n_i s^2,\end{equation} where $s \in \Z$ and $n_i$ is one of at most $O(\log^{10} X)$ squarefree integers, with $n_1 = 1$ and $n_i \geq X^{\varrho}$ if $i \geq 2$, where $\varrho > 0$ is an absolute constant. The number of $x \in [X]$ for which \eqref{false-eq} holds for a given $i$ is $\ll H\sqrt{X/n_i}$, and so the number of $x$ for which this holds for \emph{some} $i \geq 2$ is $\ll H \log^{10} X \cdot X^{(1- \varrho)/2} \ll H X^{1/2-\varrho/4}$. If $i = 1$, so that $n_1 = 1$, then $4a d x + \Delta d^2$ is a square and so $x \in \psi(\Q)$. This concludes the proof of the first bound in \eqref{claim}.

We turn now to the proof of the second bound in \eqref{claim}. 
Recall first of all that $|U_p| \leq 2\eps p$ for at least 98\% of all $p \sim X^{1/4}$, and also that for every $x \in E$ we have $x \md{p} \in U_p$ for at least 3\% of all $p \sim X^{1/4}$. For every $x \in E$, both of these events occur for at least 1\% of all $p \sim X^{1/4}$. 

Write $E_p$ for the subset of $E$ whose elements belong to $U_p$ modulo $p$. By the preceding facts we have
\[ \sum_{\substack{p \sim X^{1/4} \\ |U_p| \leq 2\eps p}} \frac{\log p}{p}|E_p|  = \sum_{x \in E} \sum_{\substack{p \sim X^{1/4} \\ |U_p| \leq 2\eps p}} \frac{\log p}{p}1_{x \md{p} \in U_p}  \geq \frac{1}{100}|E|\sum_{p \sim X^{1/4}} \frac{\log p}{p}. \]
Writing $\mathscr{P} := \{ p \sim X^{1/4} : |U_p| \leq 2\eps p  \; \text{and} \; |E_p| \geq \textstyle\frac{1}{200}|E| \}$, it follows that
\[ \sum_{p \in \mathscr{P}} \frac{\log p}{p} \geq \frac{1}{200} \sum_{p \sim X^{1/4}} \frac{\log p}{p}.\] Applying the larger sieve (that is Theorem \ref{larger-sieve}) with the choices $\delta = \frac{1}{200}$, $Q = X^{1/4}$ and $\sigma_{p} = 2\eps$, we obtain $|E| \ll X^{1/4}(\frac{1}{2^{20}\eps}\sum_{p \sim X^{1/4}} \frac{\log p}{p} - \log X)^{-1}$, provided that the term in parentheses is positive. That term is $> (2^{-21} C - 1)\log X$, which is positive if $C > 2^{22}$ (say). Thus we get the bound $|E| \ll X^{1/4}$.
This completes the proof of \eqref{claim}, and hence \eqref{weak} and Theorem \ref{stability-thm}.\end{proof}

We turn now to the proof of Theorem \ref{stability-thm-symmetric}, the stability theorem for a single infinite set $\A$. Again, we begin by recalling the statement.

\begin{stability-thm-symmetric-repeat}
Suppose that $\A$ is a set of positive integers and that $|\A \md{p}| \leq \frac{1}{2}(p+1)$ for all sufficiently large primes $p$. Then one of the following options holds:
\begin{enumerate}
\item \textup{(Quadratic structure)} There is a rational quadratic $\psi$ such that all except finitely many elements of $\A$ are contained in $\psi(\Q)$;
\item \textup{(Better than large sieve)} For each integer $k$ there are arbitrarily large values of $X$ such that $|\A[X]| < \frac{X^{1/2}}{\log^k X}$;
\item \textup{(Far from quadratic structure)} Given any rational quadratic $\psi$, for all $X$ we have $|\A[X] \cap \psi(\Q)| \leq X^{1/4 + o_{\psi}(1)}$.
\end{enumerate}
\end{stability-thm-symmetric-repeat}
\begin{proof}
Suppose that neither item (ii) nor item (iii) holds. Then there is an $\eps > 0$ and a rational quadratic $\psi$ such that $|\A[X] \cap \psi(\Q)| > X^{1/4 + \eps}$ for arbitrarily large values of $X$. For any such $X$ we may apply Theorem \ref{stability-thm} with $A = \A[X]$ to conclude that either option (ii) of our present theorem holds, or else for infinitely many $X$ we have $|\A[X] \setminus \psi(\Q)| \ll X^{1/2 - c}$.  (We note in passing that Lemma \ref{large-sieve-app} is redundant inside the proof of Theorem \ref{stability-thm} in this setting, being trivially true.)  We will deduce from this and further applications of the fact that $|\A \md{p}| \leq \frac{1}{2}(p+1)$ for $p$ sufficiently large that either (i) or (ii) of Theorem \ref{stability-thm-symmetric} holds. 
Let $\tilde \psi$ be a rational quadratic satisfying the conclusions of Lemma \ref{tedious-local-lemma}, thus $\psi(\Q) \cap \Z \subset \tilde\psi(\Z)$ and for all sufficiently large primes $p$ the reductions $\md{p}$ of $\psi(\Q) \cap \Z$ and of $\tilde\psi(\Z)$ are the same. Suppose that option (ii) of Theorem \ref{stability-thm-symmetric} does not hold for the infinitely many values of $X$ for which we have $|\A[X] \setminus \psi(\Q)| \ll X^{1/2 - c}$. Then there is some integer $k$ such that (letting $X$ range through these values)
\begin{equation}\label{weak2} \lim\sup_{X \rightarrow \infty}  X^{-1/2}\log^k X|\A[X] \cap \psi(\Q)| = \infty.\end{equation} We claim that this implies statement (i) of Theorem \ref{stability-thm}, and in fact the stronger conclusion $|\A \setminus \psi(\Q)\big| \leq k+1$. Suppose this statement is false. Then there are elements $x_1,\dots,x_{k+1}$ in $\A$ but not in $\psi(\Q)$. Since $x$ lies in $\psi(\Q)$ if and only if $4ad x + \Delta d^2$ is the square of a rational number, it follows that none of $4ad x_i + \Delta d^2$ is a square. Set $m_i := 4ad x_i + \Delta d^2$, and suppose that $p$ is a prime such that $(m_i | p ) = -1$. If $p$ is sufficiently large then $x_i \notin (\psi(\Q) \cap \Z) \md{p}$ and hence $x_i \notin \tilde\psi(\Z) \md{p}$. 

For each prime $p$, let $k(p)$ be the number of indices $i \in \{1,\dots, k+1\}$ such that $(m_i | p) = -1$. From the above reasoning and the assumption that $x_i \in \mathscr{A}$ it follows that $\A \cap \tilde\psi(\Z) \md{p}$ must occupy a set of size at most $\frac{1}{2}(p+1) - k(p)$ for all sufficiently large primes $p$. Define a set $\mathscr{B} \subset \Z$ by $\A \cap \tilde\psi(\Z) = \tilde\psi(\mathscr{B})$. Thus $|\tilde\psi(\mathscr{B}) \md{p}| \leq \frac{1}{2}(p+1) - k(p)$ for all sufficiently large primes $p$, which implies that $|\mathscr{B} \md{p}| \leq p - 2k(p) + 1$. Note also that $\A[X] \cap \psi(\Q) \subset \tilde\psi{(\mathscr{B} \cap [c_1 \sqrt{X}, c_2 \sqrt{X}])}$ for some constants $c_1,c_2$ depending only on $\tilde\psi$. We may now apply Lemma \ref{small-from-large} to the set $\mathscr{B}$. In that lemma we may take $w(p) = 2k(p) - 1$, where $k(p)$ is the number of $i$ for which $(m_i | p) = -1$, or equivalently (if $p > 2$) for which $\chi_i(p) = -1$ where $\chi_i(n) = (4m_i | n)$ is a real Dirichlet character and $( \, | \, )$ denotes the Kronecker symbol.  Thus $k(p) = \frac{1}{2}(k+1) - \frac{1}{2}\sum_{i=1}^{k+1} \chi_i(p)$, and so $w(p) = k -\sum_{i = 1}^{k+1} \chi_i(p)$. The conditions of Lemma \ref{small-from-large} are easily satisfied by the prime number theorem for characters with a fairly crude error term. It follows from Lemma \ref{small-from-large} and the above discussion that 
\[ |\mathscr{A}[X] \cap \psi(\Q)| \leq |\mathscr{B} \cap [c_1 \sqrt{X}, c_2 \sqrt{X}]| \ll X^{1/2}(\log X)^{-k},\] contrary to \eqref{weak2}. \end{proof}

\section{Composite numbers in $\A + \B$}

Recall from the introduction the following conjecture of ``inverse large sieve'' type. 

\begin{stability-conj-repeat}
Let $X_0 \in \N$, and let $\rho > 0$. Let $X \in \N$ be sufficiently large in terms of $X_0$ and $\rho$. Suppose that $A,B \subset [X]$ and that $|A \md{p}| + |B \md{p}| \leq p+1$ for all $p \in [X_0, X^{1/4}]$. Then there exists a constant $c = c(\rho) > 0$ such that one of the following holds:
\begin{enumerate}
\item \textup{(Better than large sieve)} Either $|A \cap [X^{1/2}]|$ or $|B \cap [X^{1/2}]|$ is $\leq X^{1/4 - c}$;
\item \textup{(Quadratic structure)} There are two rational quadratics $\psi_A,\psi_B$ of height at most $X^{\rho}$ such that $|A \setminus \psi_{A}(\Q)|$ and $|B\setminus \psi_{B}(\Q)| \leq X^{1/2 -c}$.
\end{enumerate}
\end{stability-conj-repeat}

Our aim in this section is to prove Theorem \ref{ils-ostmann}, which is the following statement.

\begin{ils-ostmann-repeat}
Assume Conjecture \ref{stability-conj}. Let $\A, \B$ be two sets of positive integers, with $|\A|, |\B| \geq 2$, such that $\A + \B$ contains all sufficiently large primes. Then $\A + \B$ also contains infinitely many composite numbers. 
\end{ils-ostmann-repeat}

This follows quite straightforwardly from the following fact.

\begin{lemma}\label{lem6.2}
There is $\rho > 0$ with the following property. Suppose that $\psi_{A},\psi_{B}$ are two rational quadratics of height at most $X^{\rho}$ and that $A \subset \psi_{A}(\Q) \cap [X]$ and $B \subset \psi_{B}(\Q) \cap [X]$ are sets of positive integers. Suppose that $A + B$ contains no composite number. Then at least one of $A$ and $B$ has cardinality $\ll X^{1/3}$.
\end{lemma}

\begin{proof}[Proof of Theorem \ref{ils-ostmann} given Lemma \ref{lem6.2}]
Let $\A, \B$ be two sets of positive integers with $|\A|, |\B| \geq 2$ such that $\A + \B$ coincides with the set of primes on $[X_0,\infty)$. 

We claim that there are infinitely many $X$ such that either $|\A[X]|$ or $|\B[X]|$ has cardinality at most $X^{1/2-c}$. This, however, is contrary to a theorem\footnote{The state-of-the-art here is $|\A[X]| \gg X^{1/2}/\log X \log\log X$: see \cite{elsholtz-harper}.} of Elsholtz \cite{elsholtz}, which implies that $|\A[X]|, |\B[X]| \gg X^{1/2} \log^{-5} X$ for all sufficiently large $X$. 

It remains to prove the claim. Let $\rho$ be as in Lemma \ref{lem6.2}. For each $X$, write $A := \A \cap (X^{1/4},X]$ and $B := \B \cap (X^{1/4}, X]$.  If $p \in [X_0, X^{1/4}]$ then $A + B$ contains no multiple of $p$, since any such number would be a nontrivial multiple of $p$ (and hence composite) and lies in $\A + \B$. For these primes $p$, then, we have $|A \md{p}| + |B \md{p}| \leq p$, since $B \md{p}$ cannot intersect $(-A) \md{p}$. Assuming Conjecture \ref{stability-conj}, for each $X$ one of the two options (i) or (ii) of that conjecture holds. 

If (i) holds for infinitely many $X$ then without loss of generality we have $|A \cap [X^{1/2}]| \ll X^{1/4 - c}$ for infinitely many $X$. By Elsholtz \cite{elsholtz} we have $|\A[X^{1/4}]| \ll X^{1/8 + o(1)}$, and therefore $|\A[X^{1/2}]| \ll X^{1/4 - c}$ for infinitely many $X$, thereby establishing the claim. 

Suppose, then, that (ii) holds for all sufficiently large $X$. That is, there are rational quadratics $\psi_{A}, \psi_{B}$ of height $X^{\rho}$ such that $|A\setminus \psi_{A}(\Q)|, |B \setminus \psi_{B}(\Q)| \leq X^{1/2-c}$. Write $A' := A \cap \psi_{A}(\Q)$ and $B' := B \cap \psi_{B}(\Q)$. Certainly $A' + B'$ contains no composite numbers. By Lemma \ref{lem6.2}, for all $X$ at least one of $A', B'$ has cardinality $\ll X^{1/3}$, and this means that indeed either $\A[X]$ or $\B[X]$ has size at most $X^{1/2-c}$ for infinitely many $X$.
\end{proof}

\begin{proof}[Proof of Lemma \ref{lem6.2}] Write $\psi_{A}(x) =  \frac{1}{d_A}(a_A x^2 + b_A x + c_A)$, $\psi_{B}(x) = \frac{1}{d_B}(a_B x^2 + b_B x + c_B)$. Here $a_A, a_B, b_A, b_B, c_A, c_B, d_A, d_B $ are integers, all of magnitude at most $H = X^{\rho}$. Set $Y := (H^{2}X)^{1/4}$. If $A + B$ contains no composite number then the set $(A \cap (Y, X]) + (B \cap (Y, X])$ contains no multiple of any prime $p \leq Y$. Note also that $\psi_{A}^{-1}(\Z) \subset \frac{1}{a_A}\Z$ and $\psi_{B}^{-1}(\Z) \subset \frac{1}{a_B} \Z$. Set
\[ S_A := \{ x \in \Z : \psi_{A}(\frac{x}{a_A}) \in A \cap (Y, X] \}, \quad S_B := \{ x \in \Z : \psi_{B}(\frac{x}{a_B}) \in B \cap (Y, X] \},\] and note that $\psi_{A}(\frac{x_A}{a_A}) + \psi_{B}(\frac{x_B}{a_B}) \neq 0 \md{p}$ whenever $x_A \in S_A$, $x_B \in S_B$, and $p \leq Y$ is a prime. 
To prove Lemma \ref{lem6.2}, it suffices to show that either $|S_A|$ or $|S_B|$ has size $\ll X^{1/3}$.
Note furthermore (by completing the square) that $S_A,S_B \subset [-4(H^{2}X)^{1/2}, 4(H^{2}X)^{1/2}]$.

We will focus attention only on those primes $p \leq Y$ for which $(-a_A a_B d_A d_B | p) = 1$, that is to say for which $-a_A a_B d_A d_B$ is a square modulo $p$. We look for such primes amongst the $p \leq Y$ with $p \equiv 1 \md{8}$. Since both $-1$ and $2$ are squares modulo such a prime, it certainly suffices to additionally ensure that $(q_1 |p) = 1, \dots, (q_k | p) = 1$, where $q_1,\dots, q_k$ are the distinct odd primes appearing in $a_A a_B d_A d_B$. This is equivalent to the union of $(q_1 - 1) \dots (q_k - 1)/2^k$ congruence conditions modulo $q_1 \dots q_k$. Together with the condition $p \equiv 1\md{8}$, we get the union of at least $2^{-k-2} \phi(q)$ congruence conditions modulo $q := 8 q_1 \dots q_k$. Since $|a_A|, |a_B|, |d_A|, |d_B| \leq H$ we have $q \leq 8 H^4$. Now we invoke \cite[Corollary 18.8]{iwaniec-kowalski}, a quantitative version of Linnik's theorem on the least prime in an arithmetic progression, which implies that there are at least $\frac{Y}{\phi(q)\sqrt{q}\log Y}$ primes $p \leq Y$ satisfying each of these congruence conditions, and hence $\gg  \frac{2^{-k}Y}{\sqrt{q}\log Y}$ such primes in total. (Here we used the fact that $H = X^{\rho}$ with $\rho$ sufficiently small.) Write $\mathscr{P}$ for the set of such primes, thus $|\mathscr{P}| \gg X^{1/4 - o(1)}H^{-2}$.

Now suppose that $p$ is such a prime and that $-a_A d_A/a_B d_B \equiv m^2 \md{p}$. Then we have
\begin{align*} \psi_A(\frac{x_A}{a_A}) + \psi_B (\frac{x_B}{a_B}) = \frac{1}{a_A d_A} (x_{A}^{2} + b_{A}x_{A} + a_{A}c_{A}) & + \frac{1}{a_B d_B} (x_{B}^{2} + b_{B}x_{B} + a_{B}c_{B}) \\ & \equiv \frac{1}{a_A d_A} \big(( x_A + c_1)^2 - (mx_B + c_2)^2 + c_3\big), \end{align*} where $c_1,c_2,c_3$ do not depend on $x_A, x_B$. Therefore for each prime $p \in \mathscr{P}$ we have one of the following alternatives.
\begin{enumerate}
\item $c_3 \equiv 0$ modulo $p$. Then whenever $x_A \in S_A \md{p}$ we must have $\frac{1}{m}(x_A + c_1 - c_2) \notin S_B \md{p}$, whence $|S_A \md{p}| + |S_B \md{p}| \leq p$.
\item $c_3 \not\equiv 0$ modulo $p$. Then we have $\psi_{A}(\frac{x_A}{a_A}) + \psi_B(\frac{x_B}{a_B}) \equiv 0 \md{p}$ whenever
\[ (x_A + mx_B + c_1 + c_2)(x_A - mx_B + c_1 - c_2) \equiv -c_3,\]
an equation which has $p-1$ solutions $(x_A, x_B)$. This solution set must be disjoint from $S_A \times S_B$. 
Since to each $x_A$ there are at most two $x_B$, and to each $x_B$ there are at most two $x_A$, this forces at least one of $S_A \md{p}, S_B \md{p}$ to have size $\leq 7p/8$, say. 
\end{enumerate}
In both cases at least one of $S_A \md{p}, S_B \md{p}$ has size $\leq 7p/8$. Without loss of generality the first holds for at least half the elements of $\mathscr{P}$. Finally the large sieve, as in Proposition \ref{classicls}, tells us that $|S_A| \ll \frac{(H^{2}X)^{1/2}}{|\mathscr{P}|} \ll H^3X^{1/4 + o(1)}$. If $\rho$ is small enough then this is $\ll X^{1/3}$, as required.
\end{proof}

\appendix

\section{Basic facts about rational quadratics}
\label{rat-quad-app}

\begin{lemma}\label{tedious-local-lemma}
Suppose that $\psi$ is a rational quadratic such that $\psi(\Q) \cap \Z$ is nonempty. Then there is another rational quadratic $\tilde\psi$ with $\psi(\Q) = \tilde\psi(\Q)$ such that $\psi(\Q) \cap \Z \subset \tilde\psi(\Z)$. Furthermore, for all sufficiently large primes $p$ the reductions $\mdlem{p}$ of $\psi(\Q) \cap \Z$ and of $\tilde\psi(\Z)$ are the same. 
\end{lemma}
\begin{proof}
Write $\psi(x) = \frac{1}{d}(a x^2 + bx + c)$ with $a,b,c,d \in \Z$, and simply define $\tilde\psi(x) := \psi(\frac{1}{a} x)$.  The first property, that $\psi(\Q) = \tilde\psi(\Q)$, is immediate. Moreover if $\psi(u/v)$ is an integer, with $u/v$ a rational in lowest terms, then $v | a$. It follows that $\psi(\Q) \cap \Z \subset \psi(\frac{1}{a} \Z) = \tilde\psi(\Z)$, the second required property. 

To get the last statement (about reductions mod $p$), let $x_0$ be a rational such that $\psi(x_0) \in \Z$ and write $x_0 := r/s$ in lowest terms. Then $\psi(x_0 + d\Z) \subset \Z$ (since $s | a$, as noted above), and so $\tilde\psi(ax_0 + ad\Z) \subset \Z$. Thus, writing $P \subset \Z$ for the infinite arithmetic progression $ax_0 + ad\Z$, we see that $\tilde\psi(P) \subset \psi(\Q) \cap \Z$. However for $p$ a sufficiently large prime, $P \md{p}$ is all of $\Z/p\Z$, thereby concluding the proof.
\end{proof}

\end{document}